\documentclass[12pt, twoside, english, headsepline]{article}

\usepackage[margin=30mm]{geometry}
\usepackage{tikz,graphicx}
\usepackage{amsmath,amsthm,amssymb,amsfonts, mathtools, amsbsy, dsfont}
\usepackage[utf8]{inputenx}
\usepackage{xcolor}
\usepackage{microtype}
\usepackage{hyperref}
\hypersetup{linkcolor=blue}

\newcommand*\dif{\mathop{}\!\mathrm{d}}  

\allowdisplaybreaks 

\newtheorem{theorem}{Theorem}[section] 

\theoremstyle{definition} 

\newtheorem{remark}{Remark}[section] 
\numberwithin{equation}{section} 

\providecommand{\keywords}[1]
{
  \small	
  \textbf{\textit{Keywords: }} #1
}
\providecommand{\MSC}[1]
{
  \small	
  \textit{2020 MSC: } #1   
}

\title{Random motions in $\mathbb{R}^3$ with orthogonal directions}
\author{Fabrizio Cinque$^1$ and Enzo Orsingher$^2$\\
        \small Department of Statistical Sciences, Sapienza University of Rome, Italy \\
        \small $^1$fabrizio.cinque@uniroma1.it $^2$enzo.orsingher@uniroma1.it
}

\begin{document}

\maketitle

\begin{abstract}
This paper is devoted to the detailed analysis of three-dimensional motions in $\mathbb{R}^3$ with orthogonal directions switching at Poisson times and moving with constant speed $c>0$. The study of the random position at an arbitrary time $t>0$ on the surface of the support, forming an octahedron $S_{ct}$, is completely carried out on the edges $E_{ct}$ and faces $F_{ct}$. In particular, the motion on the faces $F_{ct}$ is analysed by means of a transformation which reduces it to a three-directions planar random motion. This permits us to obtain an integral representation on $F_{ct}$ in terms of integral of products of  first order Bessel functions. The investigation of the distribution of the position $p=p(t,x,y,z)$ inside $S_{ct}$ implied the derivation of a sixth-order partial differential equation governing $p$ (expressed in terms of the products of three D'Alembert operators). A number of results, also in explicit form, concern the time spent on each direction and the position reached by each coordinates as the motion devolpes. The analysis is carried out when the incoming direction is orthogonal to the ongoing one and also when all directions can be uniformely choosen at each Poisson event. If the switches are governed by homogeneus Poisson process many explicit results are obtained.
\end{abstract} \hspace{10pt}


\MSC{Primary 60K99; 60G50}

\section{Introduction}

In the last decades, several papers focused on random motions with finite velocity both on the line and in multi-dimensional spaces and even in non-Euclidean spaces appeared. This is a wide class of stochastic processes that preserve the natural property of moving with finite speed along each direction. The prototype of these motions is the one-dimensional telegraph process which firstly appeared as a time continuous extension of simple and correlated random walks, see Goldstein \cite{G1951}. We can find some early reference about planar motions in Pearson \cite{P1906}.
\\Later, the analysis of random motions in $\mathbb{R}^2$ has been performed by many authors over the years. The case of minimal number of directions in $\mathbb{R}^2$ (meaning with three directions) was explored by Orsingher \cite{O2002} and Di Crescenzo \cite{Dc2002} (with arbitrary random steps between successive switches). First general results in $\mathbb{R}^d$ appeared in Samoilenko \cite{S2001} and further studies led to explicit results concerning minimal cyclic random motions, presented by Leorato \textit{et Al.} \cite{LLO2006} and Lachal \cite{L2006}, who provided also an integral formula for the transition density in the non-minimal case.

For planar motions, the $m$-th general equation governing the position of the motion with $m\ge 3$ arbitrary directions was obtained by Kolesnik and Turbin \cite{KT1998}. Note that in the relationship between finite-velocity random motions and partial differential equations, in the one-dimensional case, is known since the first papers and it has been strongly used to derive several results, see for instance Brooks \cite{B1999} and Orsingher \cite{O1990}. However, extracting probabilistic information by huge hyperbolic differential equations of higher order is extremely difficult. For this reasons, in the plane, two opposite assumptions proved to be fruitful, that is the case of an infinite number of directions (uniformly distributed) and the case of orthogonal directions. The first case was considered by Grosjean \cite{G1953}, Stadje \cite{S1987} and Kolesnik and Orsingher \cite{KO2005}. Later, random motions taking velocities uniformly in the continuum spectrum of possible directions where evaluated in higher spaces, see Orsingher and De Gregorio \cite{ODg2007}, Pogorui \cite{P2012} and Kolesnik \cite{K2018}.
\\The case of orthogonal directions in $\mathbb{R}^2$  was considered by Orsingher and Kolesnik \cite{OK1996}, Orsingher \cite{O2000} and recently by Orsingher \textit{et Al.} \cite{OGZ2020} in the case of cyclic movements, also in higher spaces, and in also non-homogeneous case by Cinque and Orsingher \cite{CO2021b}.

We would like to mention that in the last decade, the study of random motions with finite velocity is spreading out in the physical literature as well. We want to recall the papers of Paoluzzi \textit{et al.} \cite{PDlA2014}, Santra \textit{et al.} \cite{SBS2020}, Sevilla \cite{S2020} and Hartmann \textit{et al.} \cite{HMSS2020} concerning motion on the plane and Mori \textit{et al.} \cite{MLdMS2020} about motion in higher order spaces.

At last, other remarkable works concern fractional versions of random motions, invastigated and recently reviewed by Masoliver and Lindenberg \cite{ML2020}, Masoliver \cite{M2021} and Shitikova \cite{S2022}.
\\

Note that the classification of orthogonal planar random motion is similar to that of motions in $\mathbb{R}^3$, i.e. an orthogonal standard motion describes the position of a particle that can uniformly choose only the directions orthogonal to the current one, an orthogonal completely uniform motion describes the movement of a particle that chooses the new direction uniformly among all the possible ones (including the previous one).
\\

The object of this paper are random motions with orthogonal directions in $\mathbb{R}^3$. We study a stochastic vector process  $(X,Y,Z) = \big\{\bigl(X(t),Y(t),Z(t)\bigr)\big\}_{t\ge0}$  describing the position of a particle placed at $(x,y,z) = (0,0,0)$ at time $t=0$ and which uniformly chooses one of the six possible directions, $d_0 = (1,0,0),\ d_1= (0,1,0),\ d_2 = (0,0,1),\ d_3=(-1,0,0),\ d_4 = (0,-1,0),\ d_5 = (0,0,-1)$. The particle moves with constant velocity $c>0$ and its switches of direction are timed by a Poisson process with rate function $\lambda:(0,+\infty)\longrightarrow (0,+\infty)$. Here we consider the cases where
\begin{itemize}
\item[(1)] Orthogonal Standard Motion (OSM): from the direction $d_j$, the particle can uniformly choose among the four directions orthogonal to $d_j$ (i.e. those lying on the plane orthogonal to $d_j$), for instance if the particle moves with $d_1$ ($y$-axis), then it can choose among $d_0,d_2,d_3$ and $d_5$ (those lying on the $x,z$-plane).

\item[(2)] Orthogonal Uniform Motion (OUM): from the direction $d_j$, the particle can uniformly choose among all possible directions,including $d_j$ and its reflected direction.
\end{itemize}

The above classification follows the criterion given in \cite{CO2021b} concerning orthogonal random motion on the plane, i.e. \textit{standard} motion if the particle switches to one of the two orthogonal directions, \textit{uniform} motion if the particle uniformly switches to one of the possible directions.

Note that, in light of Remark 3.4 of \cite{CO2021b}, the orthogonal completely uniform motion is  probabilistically related to a motion where, from direction $d_j$, the particle can uniformly choose among all possible directions except $d_j$ itself. We call this Orthogonal Symmetrically Deviating Motion (OSDM). In detail, an OSDM with rate function $\lambda$ is equal in distribution to an OUM with rate function $6\lambda/5$.
\\

At time $t>0$, the support of particle position is located inside the set
\begin{equation}\label{ottaedro}
S_{ct}= \{(x,y,z)\in \mathbb{R}^3\,:,\ |x|+|y|+|z| \le ct\}.
\end{equation}
which represents an octahedron centered in the origin and with vertices placed on the coordinate axes, see Figure \ref{figuraOttaedro}.

\begin{figure}[h]
\includegraphics[scale = 0.5]{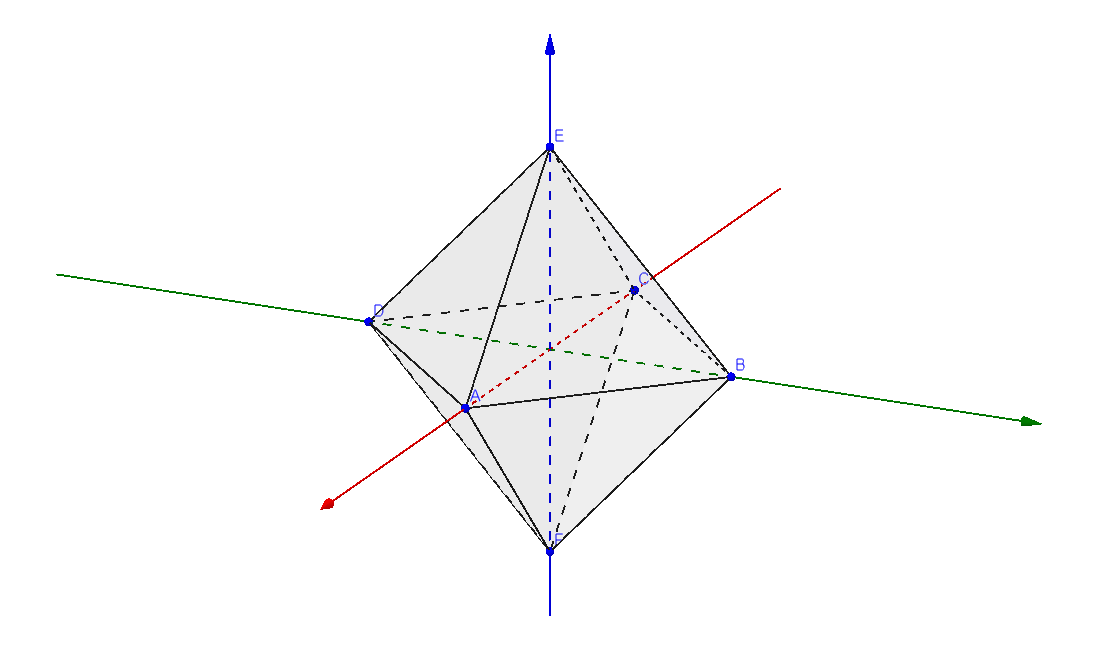}
	\caption{\small Support of orthogonal random motions in $\mathbb{R}^3$.}\label{figuraOttaedro}
\end{figure}

Our work focuses on the study of the distribution of the position $\bigl(X(t),Y(t),Z(t)\bigr)$ inside the octahedron $S_{ct}$ as well as on its surface, where the rate function is such that $\Lambda(t) = \int_0^t \lambda(s)\dif s<\infty, \ \ t>0$. If $\Lambda(t)=\infty$ the distribution of the position of the moving particle is only inside $S_{ct}$.

We show that we can distinguish three main singularities, the vertices, the edges and the faces of the polyhedron. We prove that over the edges the distribution of the process reduces to the product of a specific probability mass and the distribution of a one-dimensional symmetric random motion with two velocities, which coincides with the well-known telegraph process in the case of a homogeneous Poisson process. This result can be better understood by also proving that the probability distribution, given that the particle has never left one of the three possible Cartesian planes, coincides with the distribution of an orthogonal planar random motion (whose version depends on the version of the three-dimensional motion) studied in detailed in \cite{CO2021b}.

We also investigate the probability over the faces of the octahedron. In this case we show that its distribution is related to a planar random motion with three directions. In particular, thanks to previous works on planar motions, in particular \cite{LO2004}, we are able to provide an explicit probability density in the case of a constant rate function, for both OSM and OUM.

In the fourth section we study the marginal and joint distribution of the random times that the particle spends moving along each axis. We show that these processes are strongly related to particular cases of  two and one-dimensional random motions with finite speed and we present the explicit distribution in the case of $\lambda(t) = \lambda, \  t>0$.

In the last section is devote to the analysis of the absolutely continuous component of the motion, in particular we provide the governing equation and an integral representations for its transition probability density.

\section{Preliminaries on planar random motions}\label{sezioneMotoPianoTreDirezioni}

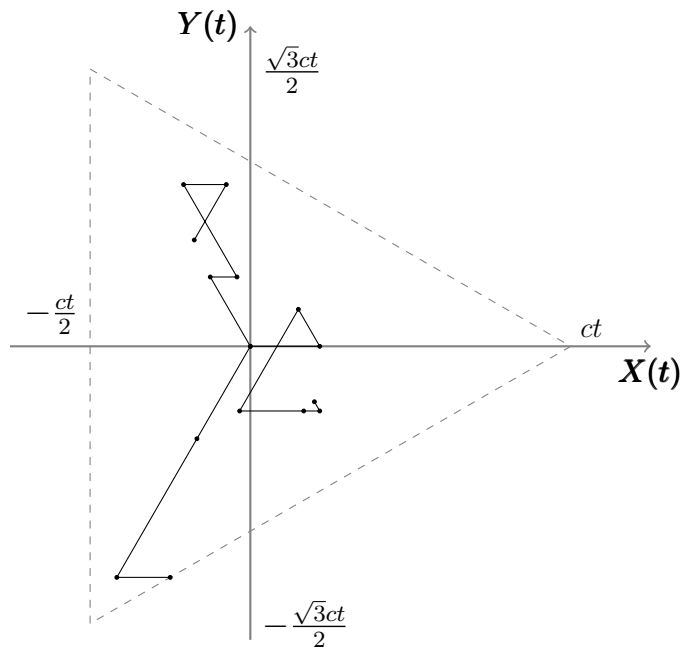
\begin{figure}
		\centering
		\begin{tikzpicture}[scale = 0.71]
		\draw[dashed, gray] (-3,5.196) -- (6,0) node[above right, black, scale = 0.9]{$ct$};
		\draw[dashed, gray] (6,0) -- (-3,-5.196) node[above left, black, scale = 0.9]{};
		\draw[dashed, gray] (-3,-5.196) -- (-3,5.196) node[below left, black, scale = 0.9]{};
		\draw (-3,0) node[above left, scale =1.1]{$-\frac{ct}{2}$};
		\draw (0,5.196) node[right, scale =1.1]{$\frac{\sqrt{3}ct}{2}$};
		\draw (0,-5.196) node[right, scale =1.1]{$-\frac{\sqrt{3}ct}{2}$};
		\draw[->, thick, gray] (-4.5,0) -- (7.5,0) node[below, scale = 1, black]{$\pmb{X(t)}$};
		\draw[->, thick, gray] (0,-5.5) -- (0,6) node[left, scale = 1, black]{ $\pmb{Y(t)}$};
		\draw (0,0)--(1.3,0)--(0.9,0.693)--(-0.2,-1.212)--(1, -1.212)--(1.3,-1.212)--(1.2,-1.0388);
		\filldraw (0,0) circle (1pt); \filldraw (1.3,0) circle (1pt); \filldraw (0.9,0.693)circle (1pt); \filldraw (-0.2,-1.212) circle (1pt); \filldraw (1,-1.212) circle (1pt); \filldraw (1.3,-1.212) circle (1pt); \filldraw (1.2,-1.0388) circle (1pt);
		\draw (0,0)--(-1,-1.732)--(-2.5,-4.33)--(-1.5, -4.33);
		\filldraw (0,0) circle (1pt); \filldraw (-1,-1.732) circle (1pt); \filldraw(-2.5,-4.33) circle (1pt); \filldraw(-1.5, -4.33) circle (1pt);
		\draw (0,0)--(-0.75,1.299)--(-0.25,1.299)--(-1.25,3.031)--(-0.45,3.031)--(-1.05,1.9918);
		\filldraw(-0.75,1.299) circle (1pt); \filldraw (-0.25,1.299) circle (1pt); \filldraw(-1.25,3.031) circle(1pt); \filldraw(-0.45,3.031)circle (1pt); \filldraw (-1.05,1.9918) circle(1pt);
		\end{tikzpicture}
		\caption{\small Sample paths of a minimal uniform random motion in $\mathbb{R}^2$ with directions (\ref{velocitaMotoPianoLO}).}\label{motoTreDirezioni}
\end{figure}

We start with some information on planar random motions with directions
\begin{equation}\label{velocitaMotoPianoLO}
v_0 = (c,0),v_1 = (-c/2,\sqrt{3}c/2),v_2 = (-c/2,-\sqrt{3}c/2) 
\end{equation}
or equivalently $v_i = (c\,\cos\frac{2i\pi}{3},c\,\sin\frac{2i\pi}{3}),\  i=0,1,2.$ At time $t>0$, a particle moving with directions (\ref{velocitaMotoPianoLO}) is located in the triangle $T_{ct} =\{(x,y)\in\mathbb{R}^2\,:\,-ct/2\le x\le ct, (x-ct)/\sqrt{3}\le y\le (ct-x)/\sqrt{3} \}$ with probability one. Among the earliest works concerning these kind of motions we cite the papers by Di Crescenzo \cite{Dc2002} and Orsingher \cite{O2002}. The paper by Leorato and Orsingher \cite{LO2004} provides a complete analysis of this planar motion and the authors obtain the explicit form of the transition density of the process, by means of order statistics, in the case of a completely uniform motion, i.e. at every Poisson event the new direction is uniformly chosen among the three possible ones (see formula (2.10) of \cite{LO2004}). 
\\

Let $\big\{\bigl(X(t),Y(t)\bigl)\big\}_{t\ge0}$ be a completely uniform planar random motion with directions $v_0,v_1,v_2$ whose changes of direction are governed by a homogeneous Poisson process with constant rate $\lambda>0$, then for $(x,y)\in \mathring{T_{ct}}$ (i.e. the absolutely continuous component of the distribution) we have that
\begin{align}
&P\{X(t)\in \dif x,Y(t)\in \dif y\}= \frac{2e^{-\lambda t}}{\sqrt{3}}\dif x\dif y \label{distribuzioneMotoPianoUniformeLO}\\
&\times\sum_{n_0=0}^\infty\sum_{n_1=0}^\infty\sum_{n_2 = 0}^\infty\Bigl(\frac{\lambda}{c}\Bigr)^{n_0+n_1+n_2+2}\frac{(n_0+n_1+n_2+3)!}{3^{2(n_0+n_1+n_2)+4}} \frac{(ct-2x)^{n_0}(ct-x+\sqrt{3}y)^{n_1}(ct-x-\sqrt{3}y)^{n_2}}{n_0!(n_0+1)!\,n_1!(n_1+1)!\,n_2!(n_2+1)!}\nonumber
\end{align}
where we have suitably reordered the sums of formula (2.10) of \cite{LO2004} in the following way
\begin{align*}
\sum_{k=2}^\infty\sum_{n_0=1}^{k-1}\sum_{n_1=1}^{k-n_0} a_{k,n_0,n_1}&= \sum_{n_0=1}^\infty\sum_{k=n_0+1}^\infty\sum_{n_1=1}^{k-n_0} a_{k,n_0,n_1}= \sum_{n_0=1}^\infty\sum_{h=1}^{\infty}\sum_{n_1=1}^{h}  a_{h+n_0,n_0,n_1}\\
&=\sum_{n_0=1}^\infty\sum_{n_1=1}^\infty\sum_{h=n_1}^{\infty} a_{h+n_0,n_0,n_1}=\sum_{n_0=1}^\infty\sum_{n_1=1}^\infty\sum_{m_2 = 0}^\infty a_{m_2+n_0+n_1,n_0,n_1}\\
& = \sum_{m_0=0}^\infty\sum_{m_1=0}^\infty\sum_{m_2=0}^{\infty} a_{m_0+m_1+m_2+2,m_0+1,m_1+1}\ .
\end{align*}

We establish an integral representation for the distribution (\ref{distribuzioneMotoPianoUniformeLO}). Let $z_0 = ct+2x,\ z_1 = ct-x+\sqrt{3}y,\ z_2 = ct-x-\sqrt{3}y$, then
\begin{align}
P\{&X(t)\in \dif x,Y(t)\in \dif y\}/(\dif x\dif y) \nonumber\\
&= \frac{2e^{-\lambda t}}{\sqrt{3}}\sum_{n_0=0}^\infty\sum_{n_1=0}^\infty\sum_{n_2 = 0}^\infty\Bigl(\frac{\lambda}{c}\Bigr)^{n_0+n_1+n_2+2}\frac{z_0^{n_0}z_1^{n_1}z_2^{n_2}}{3^{2(n_0+n_1+n_2+2)}} \frac{(n_0+n_1+n_2+3)!}{n_0!(n_0+1)!\,n_1!(n_1+1)!\,n_2!(n_2+1)!}\nonumber \\
& = \frac{2e^{-\lambda t}}{\sqrt{3}}\Bigl(\frac{\lambda}{3^2c}\Bigr)^2\sum_{n_0=0}^\infty \Bigl(\frac{\lambda z_0}{3^2c}\Bigr)^{n_0}\frac{1}{n_0!(n_0+1)!} \sum_{n_1=0}^\infty\Bigl(\frac{\lambda z_1}{3^2c}\Bigr)^{n_1}\frac{1}{n_1!(n_1+1)!} \sum_{n_2 = 0}^\infty\Bigl(\frac{\lambda z_2}{3^2c}\Bigr)^{n_2}\frac{1}{n_2!(n_2+1)!} \nonumber \\
&\ \ \ \times\int_0^\infty e^{-w}w^{n_0+n_1+n_2+3}\dif w \nonumber \\
& = \frac{2e^{-\lambda t}}{\sqrt{3}} \int_0^\infty e^{-w}w^{\frac{3}{2}}\,I_1\Bigl(\frac{2}{3}\sqrt{\frac{\lambda z_0\, w}{c}}\Bigr) \, I_1\Bigl(\frac{2}{3}\sqrt{\frac{\lambda z_1\,w}{c}}\Bigr)\,I_1\Bigl(\frac{2}{3}\sqrt{\frac{\lambda z_2\,w}{c}}\Bigr) \,\frac{1}{3}\sqrt{\frac{\lambda}{c}}\frac{1}{\sqrt{z_0z_1z_2}}\dif w \nonumber\\
& = \frac{4e^{-\lambda t}}{3\sqrt{3}} \sqrt{\frac{\lambda}{c}}\frac{1}{\sqrt{z_0z_1z_2}} \int_0^\infty e^{-u^2}u^{4} \,I_1\Bigl(\frac{2u}{3}\sqrt{\frac{\lambda z_0}{c}}\Bigr) \, I_1\Bigl(\frac{2u}{3}\sqrt{\frac{\lambda z_1}{c}}\Bigr)\,I_1\Bigl(\frac{2u}{3}\sqrt{\frac{\lambda z_2}{c}}\Bigr)\dif u, \label{distribuzioneIntegraleMotoPianoUniformeLO}
\end{align}
where $I_\nu(x) =\sum_{k=0}^\infty \bigl(\frac{x}{2}\bigr)^{2k+\nu} \frac{1}{k!\Gamma(k+1+\nu)}$ is the modified Bessel function of order $\nu \in \mathbb{R}$, with $x\in \mathbb{R}$.

\begin{remark}[Non-homogeneous Poisson process]
Note that for a completely uniform planar motion with velocities $v_0,v_1,v_2$ whose changes of direction are governed by a non-homogeneous Poisson process with rate function $\lambda\in C^2\bigl((0,\infty),(0,\infty)\bigl)$ such that $\Lambda(t)=\int_0^t\lambda(s)\dif s<\infty, \ t>0$, the transition density can be written as $q(t,x,y) = q_0(t,x,y)+ q_1(t,x,y)+ q_2(t,x,y)$, where $q_i(t,x,y)\dif x\dif y=P\{X(t)\in \dif x,Y(t)\in \dif y,D(t) = v_i\},$ $i=0,1,2,$ and which satisfy the following differential system
\begin{equation}\label{sistemaDifferenzialeMotoPianoUniformeLO}
\begin{cases}
\frac{\partial q_0}{\partial t} = -c\frac{\partial q_0}{\partial x}+\frac{\lambda(t)}{3}(q_1+q_2-2q_0),\\
\frac{\partial q_1}{\partial t} = \frac{c}{2}\frac{\partial q_1}{\partial x}-\frac{\sqrt{3}c}{2}\frac{\partial q_1}{\partial y} +\frac{\lambda(t)}{3}(q_0+q_2-2q_1),\\
\frac{\partial q_2}{\partial t} =\frac{c}{2}\frac{\partial q_1}{\partial x}+\frac{\sqrt{3}c}{2}\frac{\partial q_1}{\partial y} +\frac{\lambda(t)}{3}(q_0+q_1- 2q_2),
\end{cases}
\end{equation}
subject to the conditions $q_i\ge0$ and $\int\int_{Support\bigl(X(t),Y(t)\bigr)} q(t,x,y)\dif x\dif y =\Bigl(1-e^{-\frac{\Lambda(t)}{3}}\Bigr)^2.$
\\Clearly, if $\Lambda(t) = \infty\ \forall\ t$, the only change required is that the second member of the second boundary condition is equal to $1\ \forall\ t$.
\hfill$\diamond$
\end{remark}

\begin{remark}[Symmetrically deviating motion]\label{motoPianoTreDirezioniSD}
From the results concerning the uniform motion, we immediately obtain a complete picture of the symmetrically deviating version of the motion, i.e. when at every Poisson event the new direction is chosen uniformly among the possible directions excluding the current one. In fact, as explained in Remark 3.4 of \cite{CO2021b}, the symmetrically deviating version with rate function $\lambda$ is equal in distribution to the uniform version with rate function $3\lambda /2$ (the system corresponding to (\ref{sistemaDifferenzialeMotoPianoUniformeLO}) is exactly the same with $\lambda/2$ replacing $\lambda/3$ and the boundary conditions are also suitably modified).\hfill$\diamond$
\end{remark}

\begin{remark}[One-dimensional telegraph process]
In the same spirit of the above calculations leading to formula (\ref{distribuzioneIntegraleMotoPianoUniformeLO}) we can give an alternative form for the distribution of the absolutely continuous component of the one-dimensional symmetric telegraph process $\{\mathcal{T}(t)\}_{t\ge0}$. We recall that, for $|x|<ct$,
\begin{equation}\label{distribuzioneTelegrafoSimmetrico}
P\{\mathcal{T}(t)\in \dif x\}/\dif x = \frac{e^{-\lambda t}}{2c}\Biggl[\lambda I_0\Bigl(\frac{\lambda}{c}\sqrt{c^2t^2-x^2}\Bigr)+\frac{\partial}{\partial t}I_0\Bigl(\frac{\lambda}{c}\sqrt{c^2t^2-x^2}\Bigr) \Biggr].
\end{equation}

We show that (\ref{distribuzioneTelegrafoSimmetrico}) can be written in the following alternative ways
\begin{align}
\frac{e^{-2\lambda t}}{2}&\sum_{m=0}^\infty\sum_{n=0}^\infty\Bigl(\frac{2\lambda}{c}\Bigr)^{m+n+1}\frac{(ct-x)^{m}(ct+x)^{n}}{2^{2(m+n+1)}} \frac{(m+n+2)!}{m!(m+1)!\,n!(n+1)!} \label{distribuzioneTelgrafoSerie}\\
& =  \frac{e^{-2\lambda t}}{\sqrt{(c^2t^2-x^2)}}  \int_0^\infty e^{-w^2}w^{3} \, I_1\Bigl(w\sqrt{\frac{2\lambda}{c}(ct-x)}\Bigr)\,I_1\Bigl(w\sqrt{\frac{2\lambda}{c}(ct+x)}\Bigr)\dif w. \label{distribuzioneTelegrafoIntegrale}
\end{align}
The above equality follows by proceeding in the same way as in the derivation of formula (\ref{distribuzioneIntegraleMotoPianoUniformeLO}).
\\
\\
In order to prove the equality between (\ref{distribuzioneTelegrafoSimmetrico}) and (\ref{distribuzioneTelegrafoIntegrale}), we recall that for $\nu,w,A\in\mathbb{R},\ I_\nu(Aw)$ satisfies
\begin{equation}\label{equazioneDiBessel}
\frac{\dif^2 f}{\dif w^2}+\frac{1}{w}\frac{\dif f}{\dif w}-\Bigl(A^2+\frac{\nu^2}{w^2}\Bigr)f = 0\ \ \text{and} \ \ \frac{\dif }{\dif w}I_0(Aw) = AI_1(Aw).
\end{equation}
We now consider the following relationship, with $A,B$ being arbitrary real numbers,
\begin{align}
\int_0^\infty &w^{3}e^{-w^2} \, I_1(Aw)\,I_1(Bw)\dif w \\
& = \frac{1}{AB}\int_0^\infty e^{-w^2}w^{3} \,\frac{\partial }{\partial w} I_0(Aw)\,\frac{\partial }{\partial w}I_0(Bw)\dif w \nonumber\\
& = \frac{1}{AB} \Bigl[-\frac{w^2e^{-w^2}}{2} \frac{\partial }{\partial w}I_0(Aw)\frac{\partial }{\partial w}I_0(Bw)\Bigr]_0^\infty + \frac{1}{2AB} \int_0^\infty e^{-w^2} \Biggl(2w \frac{\partial }{\partial w}I_0(Aw)\frac{\partial }{\partial w}I_0(Bw)\nonumber\\
&\ \ \ +w^2 \frac{\partial^2 }{\partial w^2}I_0(Aw)\frac{\partial }{\partial w}I_0(Bw) + w^2\frac{\partial }{\partial w}I_0(Aw)\frac{\partial^2 }{\partial w^2}I_0(Bw) \Biggr)\dif w\nonumber\\
& = \frac{1}{2AB}\int_0^\infty w^2e^{-w^2} \Bigl[ A^2 I_0(Aw)\,\frac{\partial }{\partial w}I_0(Bw) + B^2 \frac{\partial }{\partial w}I_0(Aw)\,I_0(Bw) \Bigr]\dif w \nonumber\\
& =\frac{A^2+B^2}{4} \int_0^\infty we^{-w^2}I_1(Aw)I_1(Bw)\dif w + \frac{AB}{2}\int_0^\infty we^{-w^2}I_0(Aw)I_0(Bw)\dif w\label{ultimoPassaggioIntegralePerTelegrafoUniforme}
\end{align}
where in the second-last equality we used the Bessel equation in (\ref{equazioneDiBessel}) and in the last equality we integrated by parts, used again (\ref{equazioneDiBessel}) and performed some simple algebra.
\\By means of formula (6.633) of Gradshteyn-Ryzhik \cite{GR1980} we obtain that, for $\nu,A,B\in \mathbb{R}, \alpha>0$,
\begin{equation}
\int_0^\infty we^{-\alpha w^2}I_\nu(Aw)I_\nu(Bw)\dif w= \frac{e^{\frac{A^2+B^2}{4\alpha}}}{2\alpha}I_\nu\Bigl(\frac{AB}{2\alpha}\Bigr).
\end{equation}
Thus, formula (\ref{ultimoPassaggioIntegralePerTelegrafoUniforme}) becomes
\begin{equation}\label{integraleConProdottoI1}
\int_0^\infty w^{3}e^{-w^2} \, I_1(Aw)\,I_1(Bw)\dif w=\frac{e^{\frac{A^2+B^2}{4}}}{4}\Biggl[\frac{A^2+B^2}{2} I_1\Bigl(\frac{AB}{2}\Bigr) + AB\, I_0\Bigl(\frac{AB}{2}\Bigr)\Biggr].
\end{equation}
By suitably applying (\ref{integraleConProdottoI1}) to (\ref{distribuzioneTelegrafoIntegrale}) and by using the second relationship in (\ref{equazioneDiBessel}) once again, we obtain (\ref{distribuzioneTelegrafoSimmetrico}).
\\

It is well-known that a uniform motion with constant rate $\lambda$ is equal in distribution to a classical motion with constant rate $\lambda/2$ (this is also a particular case of the result stated in Remark 3.4 of \cite{CO2021b}). Now, formula (\ref{distribuzioneTelgrafoSerie}) can be obtained by studying a completely uniform telegraph process $\{\mathcal{T}(t)\}_{t\ge0}$, i.e. when at every change of direction the particle can either switch or continue with the same velocity. In this case we can write, with $|x|<ct$,
\begin{align} 
P&\{\mathcal{T}(t)\in \dif x\}\label{distribuzioneTelegrafoUniforme}\\
&= \sum_{k=1}^\infty\sum_{m=1}^k P\{N(t) = k\} P\{N_+(t) = m\,|\,N(t)=k\}P\{\mathcal{T}(t)\in\dif x\,|\,N_+(t) = m,N(t) = k\}\nonumber
\end{align}
where $N_+(t)$ denotes the number of positive displacements performed in the time interval $[0,t]$. The middle factor is the probability mass of a $Binomial(k+1,1/2)$ random variable. The third probability can be computed by means of several strategies, see \cite{LO2004} for an equivalent two dimensional case, and for $|x|<ct,\ 1\le m\le k,$ it reads
\begin{equation}
P\{\mathcal{T}(t)\in\dif x\,|\,N_+(t) = m,N(t) = k\} = \frac{k!}{(m-1)!(k-m)!}\frac{(ct-x)^{m-1}(ct+x)^{k-m}}{(2ct)^k}\dif x.
\end{equation}
The probability (\ref{distribuzioneTelegrafoUniforme}), with rate $2\lambda$, easily leads to (\ref{distribuzioneTelgrafoSerie}).\hfill$\diamond$
\end{remark}

\section{Orthogonal motions in $\mathbb{R}^3$: distributions on the singular components}

Let $(X,Y,Z) =\big\{\bigl(X(t),Y(t),Z(t)\bigr)\big\}_{t\ge0}$ be an orthogonal random motion in $\mathbb{R}^3$. We have already explained that at time $t>0$ the position of the particle, described by the stochastic vector process $\bigl(X(t),Y(t),Z(t)\bigr)$, is in the set $S_{ct}$, given in formula (\ref{ottaedro}), see Figure \ref{figuraOttaedro}. 
\\In the case of a rate function $\lambda:(0,+\infty)\longrightarrow (0,+\infty)$ such that $\Lambda(t)=\int_0^\infty \lambda(s)\dif s<\infty, t>0$, we must distinguish among three main singularities which compose the surface $\partial S_{ct}$ of the octahedron $S_{ct}$:
\begin{itemize}
\item the (six) vertices, $V_{ct} = \{(\pm ct,0,0), (0,\pm ct, 0),(0,0,\pm ct)\}$. The particle lies in $V_{ct}$ if it takes only one direction in the time interval $[0,t]$;
\item the (twelve) edges, $\mathring E_{ct}=E_{ct}\setminus V_{ct} = \{(x,y,z)\in \mathbb{R}^3\,:\, |x|+|z|=ct \ \text{or}\ |x|+|y|=ct\ \text{or}\ |y|+|z| = ct\} \setminus V_{ct}$. The particle lies in $\mathring E_{ct}$ if it takes only two directions in the time interval $[0,t]$;
\item the (eight) faces, $\mathring F_{ct} =F_{ct}\setminus E_{ct} = \{(x,y,z)\in \mathbb{R}^3\,:\, |x|+|y|+|z|=ct\}\setminus E_{ct}$. The particle lies in $\mathring F_{ct}$ if it takes only three directions in the time interval $[0,t]$.
\end{itemize}
Note that $F_{ct} = \partial S_{ct}$. 
\\

Below we list the probability masses of all these singularities for both the orthogonal standard motion (OSM) and the orthogonal uniform motion (OUM).

Let $\big\{\bigl(X(t),Y(t),Z(t)\bigr)\big\}_{t\ge0}$ be an OSM with rate function $\lambda$ such that $\Lambda(t)=\int_0^\infty \lambda(s)\dif s <\infty\ \forall\ t$. We have that
\begin{align}
&P\big\{ \bigl(X(t),Y(t),Z(t)\bigr) \in V_{ct} \big\} = P\{N(t)=0\} = e^{-\Lambda(t)},\label{probabilitaVerticiOSM}\\
&P\big\{\bigl(X(t),Y(t),Z(t)\bigr)\in \mathring E_{ct}\big\} = \sum_{n=1}^\infty P\{N(t)=n\}\Bigl(\frac{1}{4}\Bigr)^{n-1} = 4\Bigl( e^{-\frac{3\Lambda(t)}{4}} - e^{-\Lambda(t)} \Bigr),\label{probabilitaSpigoliOSM}\\
&P\big\{\bigl(X(t),Y(t),Z(t)\bigr)\in \mathring F_{ct} \big\} = \sum_{n=2}^\infty P\{N(t)=n\} \sum_{k=0}^{n-2}\Bigl(\frac{1}{4}\Bigr)^{k}\, \frac{2}{4}\, \Bigl(\frac{2}{4}\Bigr)^{n-k-2} = 4\Bigl( e^{-\frac{\Lambda(t)}{2}} - e^{-\frac{\Lambda(t)}{4}}\Bigr)^2.\label{probabilitaFacceOSM}
\end{align}
The probability (\ref{probabilitaFacceOSM}) requires some considerations. The explicit formula takes into account that the particle must move with three different directions only. In the first two displacements two different directions are selected. Then, the particle keeps on alternating these two directions $k$ times and then it changes into one of the two possible directions concerning the ``third dimension''. It continues alternating these three selected directions for the remaining switches, meaning that each time it can choose between two directions out of four.
\\

Let $\big\{\bigl(X(t),Y(t),Z(t)\bigr)\big\}_{t\ge0}$ be an OUM with rate function $\lambda$ as above.
\begin{align}
P\big\{ \bigl(X(t),Y(t),Z(t)\bigr) \in V_{ct} \big\} &=6\sum_{n=0}^\infty P\{N(t)=n\}\Bigl(\frac{1}{6}\Bigr)^{n+1} = e^{-\frac{5\Lambda(t)}{6}}, \label{probabilitaVerticiOCUM}\\ 
P\big\{\bigl(X(t),Y(t),Z(t)\bigr)\in \mathring E_{ct}\big\} &= 12 \sum_{n=1}^\infty P\{N(t)=n\}\sum_{k=1}^{n}\binom{n+1}{k} \Bigl(\frac{1}{6}\Bigr)^k\Bigl(\frac{1}{6}\Bigr)^{n+1-k}\nonumber\\
& = 4\Bigl( e^{-\frac{2\Lambda(t)}{3}} - e^{-\frac{5\Lambda(t)}{6}} \Bigr), \label{probabilitaSpigoliOCUM}\\
P\big\{\bigl(X(t),Y(t),Z(t)\bigr)\in \mathring F_{ct} \big\} &= 8 \sum_{n=2}^\infty P\{N(t)=n\}\sum_{\substack{k_1,k_2,k_3=1,\\ k_1+k_2+k_3= n+1}}^{n-1}\binom{n+1}{k_1,k_2,k_3} \Bigl(\frac{1}{6}\Bigr)^{k_1}\Bigl(\frac{1}{6}\Bigr)^{k_2}\Bigl(\frac{1}{6}\Bigr)^{k_3}\nonumber\\
&= 4\Bigl(e^{-\frac{\Lambda(t)}{2}} -2 e^{-\frac{2\Lambda(t)}{3}}+ e^{-\frac{5\Lambda(t)}{6}}\Bigr).\label{probabilitaFacceOCUM}
\end{align}
All these probabilities are computed by multiplying the probability of reaching a precise singularity (vertex, edge, face) and their number (6 vertices, 12 edges, 8 faces).

\subsection{Distribution on the edges of the support}

The particle reaches the edge if it chooses only two directions up to time $t$. For example, if we consider the directions $d_0$ and $d_1$ only, the particle is located on the edge
\begin{equation}\label{spigoloOttaedro}
E_{ct}^{(z)} = \big\{(x,y,z)\in\mathbb{R}^3\,:\,x+y=ct,\, z=0\big\}.
\end{equation}
Here, we study the distribution of the particle lying on the edge $E_{ct}^{(z)}$.

The choice of a precise edge is non-restrictive because of symmetry properties of both the OSM and the OUM. In particular, it is important to observe that the probability (density) of being at time $t>0$ in a point $(x,y,z)\in S_{ct}$ is equal to the probability (density) of being at time $t$ in any point of kind $(\pm x, \pm y,\pm z)$. This symmetry property permits us to focus our analysis, here and in the forthcoming sections, on the portion of the support concerning non-negative coordinates, that is $S_{ct}\cap \{(x,y,z\in\mathbb{R}^3\,:\,x,y,z\ge0)\}$.
\\

We want to evaluate the probability density in an arbitrary point of $\mathring E_{ct}^{(z)}$, that is $P\{X(t)\in \dif x, X(t)+Y(t)=ct, Z(t)=0\}$ with $x\in (0,ct)$. Instead of focusing on this probability, it is convenient to study the following one, for $|v|<ct$, 
\begin{align}
p(t,v) \dif v&= P\{X(t)-Y(t)\in \dif v, X(t)+Y(t)=ct, Z(t)=0\} \label{leggeSpigolo}\\
& = P\{X(t)-Y(t)\in \dif v, X(t)+Y(t)=ct, Z(s)=0 \text{ for } s\in [0,t]\}.\nonumber
\end{align}
Let $(X,Y,Z)$ be a OSM. The particle alternates the directions $d_0$ and $d_1$ in the time interval $[0,t]$ in order to be at time $t$ on $\mathring E_{ct}^{(z)}$. We consider the densities $f_i(t,v)\dif v = P\{X(t)-Y(t)\in \dif v, X(t)+Y(t)=ct, Z(t)=0, D(t)=d_i\},\ i=0,1,$ then $p(t,v) = f_0(t,v)+f_1(t,v).$ The functions $f_i$ satisfy
\begin{equation}\label{sistemaNonDifferenzialeInizialeSpigolo}
\begin{cases}
f_0(t+\dif t,v) = f_0(t,v-c\dif t) \bigl(1-\lambda(t)\dif t\bigr) + f_1(t,v+c\dif t)\frac{\lambda(t)\dif t}{4} + o(\dif t),\\
f_1(t+\dif t,v) = f_1(t,v+c\dif t) \bigl(1-\lambda(t)\dif t\bigr) + f_0(t,v-c\dif t)\frac{\lambda(t)\dif t}{4} + o(\dif t),\\
\end{cases}
\end{equation}
and in differential form 
\begin{equation}\label{sistemaInizialeSpigolo}
\begin{cases}
\frac{\partial f_0}{\partial t} = -c\frac{\partial f_0}{\partial v}+\frac{\lambda(t)}{4}(f_1-4f_0),\\
\frac{\partial f_1}{\partial t} = c\frac{\partial f_1}{\partial v}+\frac{\lambda(t)}{4}(f_0-4f_1).
\end{cases}
\end{equation}
By considering the usual transformation $p = f_0+f_1$ and $w = f_0-f_1$ and performing some calculation we easily extract the second-order differential problem
\begin{equation}\label{sistemaInizialeSpigolo}
\begin{cases}
\frac{\partial^2 p}{\partial t^2} +2\lambda(t)\frac{\partial p}{\partial t} +\frac{3}{4}\bigl( \lambda'(t) +\frac{5}{4}\lambda(t)^2\bigr)p = c^2\frac{\partial^2 p}{\partial v^2},\\
p(t,v)\ge0,\ \ \ \int_{-ct}^{ct} p(t,v)\dif v= \frac{1}{3}\bigl(e^{-3\Lambda(t)/4}-e^{-\Lambda(t)}\bigr).
\end{cases}
\end{equation}
The third relationship easily follows from probability (\ref{probabilitaSpigoliOSM}) and by keeping in mind that the octahedron has twelve edges.
\\
We note that 
\begin{equation}\label{probabilitaPianoXY}
P\{Z(s)=0,\ s\in[0,t]\} = \frac{4}{6}\sum_{n=0}^\infty P\{N(t)=n\}\Bigl(\frac{2}{4}\Bigr)^n = \frac{2}{3}e^{-\frac{\Lambda(t)}{2}}.
\end{equation}
In light of (\ref{probabilitaPianoXY}), we consider $p(t,v)=2e^{-\Lambda(t)/2}/3 \ q(t,v)$, where $q$ can be interpreted as the conditional probability measure given $Z(s)=0$ for $s\in [0,t]$, meaning that the particle moves on the plane $x,y$ in the time interval $[0,t]$. The system (\ref{sistemaInizialeSpigolo}) becomes
\begin{equation}\label{sistemaFinaleSpigolo}
\begin{cases}
\frac{\partial^2 q}{\partial t^2} +\lambda(t)\frac{\partial q}{\partial t} +\frac{1}{4}\bigl( \lambda'(t) +\frac{3}{4}\lambda(t)^2\bigr)q = c^2\frac{\partial^2 q}{\partial v^2},\\
q(t,v)\ge0,\ \ \ \int_{-ct}^{ct} q(t,v)\dif v= \frac{1}{2}\bigl(e^{-\Lambda(t)/4}-e^{-\Lambda(t)/2}\bigr),
\end{cases}
\end{equation}
which coincides with problem (2.12) of \cite{CO2021b} with $\lambda(t)/2$ instead of $\lambda(t)$. Therefore, the distribution on the edges of an OSM, with respect to the probability measure conditioned on $\{Z(s)=0,\ s\in[0,t]\}$, is equal to the distribution on the edges of an orthogonal standard planar random motion with rate function $\lambda(t)/2, t>0$. This implies that $q$ can be expressed as the product of the probability that a particle with standard orthogonal movements on the plane reaches the edge, $e^{-\Lambda(t)/4}/2 $, and the density of a one-dimensional symmetric telegraph process with rate function $\lambda(t)/4, t>0$.
\\For $\lambda(t)=\lambda>0 \ \forall \ t$, distribution (\ref{leggeSpigolo}) reads
$$p(t,v) =  \frac{e^{-\lambda t}}{6c} \Biggl[\frac{\lambda}{4} I_0\Bigl(\frac{\lambda}{4c}\sqrt{c^2t^2-v^2}\Bigr)+\frac{\partial}{\partial t}I_0\Bigl(\frac{\lambda}{4c}\sqrt{c^2t^2-v^2}\Bigr) \Biggr].$$

The above result, suitably adapted, holds in the case of an OUM as well. In this case we must keep in mind that 
\begin{equation}\label{probabilitaPianoXYOCUM}
P\{Z(s)=0,\ s\in[0,t]\} = \frac{4}{6}\sum_{n=0}^\infty P\{N(t)=n\}\Bigl(\frac{4}{6}\Bigr)^n = \frac{2}{3}e^{-\frac{\Lambda(t)}{3}}.
\end{equation}

For both versions of the three-dimensional motion, we can obtain an even stronger result which connects the orthogonal motions in $\mathbb{R}^2$ and in $\mathbb{R}^3$.

\begin{theorem}\label{teoremaMotoSuPianoXY}
Let  $(X,Y,Z) = \big\{\bigl(X(t),Y(t),Z(t)\bigr)\big\}_{t\ge0}$ an orthogonal random motion in $\mathbb{R}^3$ with rate function $\lambda\in C^2\bigl((0,\infty),(0,\infty)\bigl)$ such that $\Lambda(t) = \int_0^t\lambda(s)\dif s<\infty, t>0$. Let $P_{xy}(\cdot)=P\{\cdot \,|\,Z(s) =0, s\in[0,t] \}$ be a conditional probability measure.
\item[$(i)$] If $(X,Y,Z)$ is an OSM, then, with respect to $P_{xy}$ it is an orthogonal standard motion on the $(x,y)$-plane with rate function $\lambda/2$;
\item[$(ii)$] if $(X,Y,Z)$ is an OUM, then, with respect to $P_{xy}$ it is an orthogonal uniform motion on the $(x,y)$-plane with rate function $\lambda/2$.
\end{theorem}
We recall that in \cite{CO2021b} the authors proved that an orthogonal standard planar stochastic motion $\big\{\bigl(X(t),Y(t)\bigr)\big\}_{t\ge0}$ with rate function $\lambda\in C^2\bigl((0,\infty),(0,\infty)\bigl)$ can be expressed as
\begin{equation}\label{decomposizioneMotoOrtogonaleStandardPiano}
\begin{cases}
X(t) = U(t) +V(t),\\
Y(t) = U(t)-V(t),
\end{cases}
\end{equation}
where $U$ and $V$ are two independent one-dimensional telegraph processes with rate function $\lambda/2$ and velocity $c/2$.

Clearly, Theorem \ref{teoremaMotoSuPianoXY} equivalently holds if we consider that the motion develops on the $(x,z)$-plane or the $(y,z)$-plane for the time interval $[0,t]$.

\begin{proof}
First of all we observe that the particle lies on the $(x,y)$-plane if it moves with the directions $d_0,d_1,d_3$ and $d_4$ only.

We consider the case $(i)$. Let $Q_{ct} =\{(x,y,z)\in\mathbb{R}^3\,:\,|x|+|y|\le ct,\, z=0\}$ the support of the position of the particle with respect to $P_{xy}$. 
\\The probability concerning the vertices of $Q_{ct}$ easily follows by taking into account (\ref{probabilitaVerticiOSM}) and (\ref{probabilitaPianoXY}). We proved above the result regarding the density on the edges.
\\

We have to study the probability density inside the square $Q_{ct}$, that is, for $(x,y,z=0)\in\mathring Q_{ct}$, $p(t,x,y) \dif x\dif y = P\{X(t)\in \dif x,Y(t)\in \dif y,Z(s)=0 \text{ for } s\in[0,t] \}$. Let $f_i(t,x,y)\dif x\dif y = P\{X(t)\in \dif x,Y(t)\in \dif y,Z(s)=0 \text{ for } s\in[0,t], D(t)=d_i \}, i=0,1,3,4$. Clearly $p(t,x,y) = \sum_{i\in\{0,1,3,4\}} f_i(t,x,y)$ and it is easy to show that the probabilities $f_i$ satisfy the differential system
\begin{equation}\label{sistemaInizialeMotoSuPianoXY}
\begin{cases}
\frac{\partial f_0}{\partial t} = -c\frac{\partial f_0}{\partial x}+\frac{\lambda(t)}{4}(f_1+f_4-4f_0),\ \ \ \frac{\partial f_3}{\partial t} = c\frac{\partial f_3}{\partial x}+\frac{\lambda(t)}{4}(f_1+f_4-4f_3),\\
\frac{\partial f_1}{\partial t} = -c\frac{\partial f_1}{\partial y}+\frac{\lambda(t)}{4}(f_0+f_3-4f_1),\ \ \ \frac{\partial f_4}{\partial t} =c\frac{\partial f_4}{\partial y}+\frac{\lambda(t)}{4}(f_0+f_3-4f_4),
\end{cases}
\end{equation}
subject to $f_i(t,x,y)\ge 0\ \forall\ i$ and (keep in mind (\ref{probabilitaPianoXY}), (\ref{probabilitaVerticiOSM}) and (\ref{probabilitaSpigoliOSM}))
\begin{align} 
\int \int_{Q_{ct}} p(t,x,y) \dif x\dif y &= \frac{2}{3}e^{-\frac{\Lambda(t)}{2}} - \Bigl(\frac{2}{3}e^{-\Lambda(t)} + \frac{4}{12}\,4\bigl(e^{-\frac{3\Lambda(t)}{4}} -e^{-\Lambda(t)}\bigr)\Bigr) \nonumber\\
&= \frac{2}{3}e^{-\frac{\Lambda(t)}{2}}\Bigl(1-e^{-\frac{\Lambda(t)}{4}} \Bigr)^2. \label{condizioneProblemaMotoSuPianoXY}
\end{align}
In light of (\ref{condizioneProblemaMotoSuPianoXY}) we set $ f_i(t,x,y) = \frac{2}{3} e^{-\frac{\Lambda(t)}{2}} q_i(t,x,y),\ \ i=0,1,3,4,$ and $q(t,x,y) = \sum_{i\in\{0,1,3,4\}} q_i(t,x,y)$. Now, with (\ref{probabilitaPianoXY}) at hand, we have that 
\begin{align*}
P\{X(t)\in \dif x,Y(t)\in \dif y\,|\,Z(s)=0 \text{ for } s\in[0,t] \}=q(t,x,y) \dif x\dif y.
\end{align*}
The differential problem (\ref{sistemaInizialeMotoSuPianoXY}) with respect to functions $q_i$ reads
\begin{equation}\label{sistemaFinaleMotoSuPianoXY}
\begin{cases}
\frac{\partial q_0}{\partial t} = -c\frac{\partial q_0}{\partial x}+\frac{\lambda(t)}{4}(q_1+q_4-2q_0),\\
\frac{\partial q_1}{\partial t} = -c\frac{\partial q_1}{\partial y}+\frac{\lambda(t)}{4}(q_0+q_3-2q_1),\\
\frac{\partial q_3}{\partial t} = c\frac{\partial q_3}{\partial x}+\frac{\lambda(t)}{4}(q_1+q_4-2q_3),\\
\frac{\partial q_4}{\partial t} =c\frac{\partial q_4}{\partial y}+\frac{\lambda(t)}{4}(q_0+q_3-2q_4),
\end{cases}\text{ s.t. }  \int \int_{Q_{ct}} q(t,x,y) \dif x\dif y = \Bigl(1-e^{-\frac{\Lambda(t)}{4}} \Bigr)^2.
\end{equation}
The problem (\ref{sistemaFinaleMotoSuPianoXY}) coincides with the problem related to the derivation of the distribution of an orthogonal standard planar random motion with rate function $\lambda/2$ (see (2.3),(2.10) of \cite{CO2021b}). This concludes the proof of $(i)$.
\\

The proof of $(ii)$, concerning the OUM, works in the same manner. 
\end{proof}

Thanks to Theorem \ref{teoremaMotoSuPianoXY} and the results of \cite{CO2021b} (Remark 2.4 and Theorem 2.2), we easily obtain the explicit probabilities for some particular rate functions. For instance, if $\lambda(t) = \lambda\ \forall\ t$ and an OSM, we have that, for $|v|<ct$,
\begin{align*}
 P\{&X(t)-Y(t)\in \dif v, X(t)+Y(t)=ct, Z(t)=0\}/\dif v \\
 & = \frac{e^{-\lambda t}}{6c}\Biggl[\frac{\lambda}{4} I_0\Bigl(\frac{\lambda}{4c}\sqrt{c^2t^2-v^2}\Bigr)+\frac{\partial}{\partial t}I_0\Bigl(\frac{\lambda}{4c}\sqrt{c^2t^2-v^2}\Bigr) \Biggr],
\end{align*}
and $x,y$ such that $|x|+|y|<ct$ (see also (\ref{decomposizioneMotoOrtogonaleStandardPiano})),
\begin{align*}
P\{X(t)\in& \dif x, Y(t) \in \dif y, Z(s)=0\text{ for }s\in [0,t]\}/(\dif x\dif y)\\
&=\frac{e^{-\lambda t}}{3c} \Biggl[\frac{\lambda}{4} I_0\Bigl(\frac{\lambda}{8c}\sqrt{c^2t^2-(x+y)^2}\Bigr)+\frac{\partial}{\partial t}I_0\Bigl(\frac{\lambda}{8c}\sqrt{c^2t^2-(x+y)^2}\Bigr) \Biggr]\\
&\ \ \ \times\Biggl[\frac{\lambda}{4} I_0\Bigl(\frac{\lambda}{8c}\sqrt{c^2t^2-(x-y)^2}\Bigr)+\frac{\partial}{\partial t}I_0\Bigl(\frac{\lambda}{8c}\sqrt{c^2t^2-(x-y)^2}\Bigr) \Biggr].
\end{align*}

\subsection{Distribution on the faces of the octahedron}\label{sottosezioneDistribuzioneFaccia}

We are now interested in studying the distribution of the motion within a face of the octahedron $S_{ct}$. Without any loss of generality, let us consider the face
\begin{equation}
F^+_{ct} =\{(x,y,z)\in \mathbb{R}^3\,:\,  x,y,z\ge0,\ x+y+z = ct\}.
\end{equation}
At time $t>0$ the particle lies on the face $F^+_{ct}$ if it moved alternating the directions $d_0,d_1$ and $d_2$ only. In particular, it is located inside the face if and only if at least one displacement with each of these three directions has been performed. The main object of this section is the derivation of the probability density
\begin{equation}
p(t,x,y)\dif x\dif y = P\{ X(t)\in\dif x, Y(t)\in \dif y, X(t)+Y(t)+Z(t)=ct\},
\end{equation}
for $(x,y,ct-x-y)\in \mathring F^+_{ct}$.
\\

The transformation
\begin{equation}\label{trasformazioneFaccia}
\begin{cases}
U(t) = X(t)-\frac{1}{2}\bigl(Y(t)+Z(t)\bigr),\\
V(t) = \frac{\sqrt{3}}{2}\bigl(Y(t)-Z(t)\bigr),\\
W(t) = X(t) +Y(t) + Z(t),
\end{cases}
\end{equation}
produces a new motion in a three-dimensional space with coordinates $u,v,w$. The triangular face $F^+_{ct}$ of the octahedron is transformed into a triangle $F_{ct}'$ that lies on the plane $w = ct$. In particular, the point $(x=ct,y=0,z=0)$ is mapped into the vertex $(u=ct,v=0,w=ct)$, the point $(x=0,y=ct,z=0)$ is mapped into the vertex $(u=-ct/2,v=\sqrt{3}ct/2,w=ct )$ and the point $(x=0,y=0,z=ct)$ is mapped into the vertex $(u=-ct/2,v=-\sqrt{3}ct/2,w=ct)$.
\\

The motion $\big\{\bigl(U(t),V(t),W(t)\bigr)\big\}_{t\ge0}$ on the face $F_{ct}'$ is similar to the random motion on a plane with directions $v_0 = (c,0),v_1 = (+c/2,\sqrt{3}c/2), v_2 = (-c/2,\sqrt{3}c/2)$, see (\ref{velocitaMotoPianoLO}), described in Section \ref{sezioneMotoPianoTreDirezioni}. We now prove that there exists a very strong relationship between these two motions.

Note that the third coordinate, $W(t)$, is almost surely equal to $ct\ \forall \ t$ (so the third coordinate of the triplet $(U,V,W)$ is deterministic). Now, assume that the original motion is moving with direction $d_0$ at time $t$, then $\dif x = X(t+\dif t)-X(t) = c\dif t,\ \dif y= Y(t+\dif t)-Y(t) = 0,\ \dif z= Z(t+\dif t)- Z(t) = 0$ and by keeping in mind the transformation (\ref{trasformazioneFaccia}), $\dif u = U(t+\dif t)-U(t) = c\dif t,\ \dif v = 0$. Similarly, if the original motion devolops with direction $d_1$ then $\dif u = -c/2\dif t$ and $\dif v = \sqrt{3}c/2\dif t$, while the displacements with direction $d_2$ lead to $\dif u=-c/2\dif t$ and $\dif v= -\sqrt{3}c/2 \dif t$. Therefore, we can relate the directions of the three-dimensional motion, $d_i$, with the directions of the planar motion, $v_i$, for $i=0,1,2$.

Let us assume $(X,Y,Z)$ being an OSM. In light of the above considerations we provide the distribution of the vector process $\big\{\bigl(U(t),V(t),W(t)\bigr)\big\}_{t\ge0}$ on the face $F_{ct}'$, $f(t,u,v) \dif u\dif v = P\{U(t)\in \dif u, V(t)\in \dif v,W(t) = ct\}$. Let us consider the joint densities $f_i(t,u,v)\dif u\dif v = P\{U(t)\in \dif u, V(t)\in \dif v,W(t) = ct, D(t) = d_i\}, i=0,1,2$, then $f = f_0+f_1+f_2$. It is easy to show that these functions satisfy the following system
\begin{equation}\label{sistemaInizialeFacciaTrasformata}
\begin{cases}
\frac{\partial f_0}{\partial t} = -c\frac{\partial f_0}{\partial u}+\frac{\lambda(t)}{4}(f_1+f_2-4f_0),\\
\frac{\partial f_1}{\partial t} = \frac{c}{2}\frac{\partial f_1}{\partial u}-\frac{\sqrt{3}c}{2}\frac{\partial f_1}{\partial v} +\frac{\lambda(t)}{4}(f_0+f_2-4f_1),\\
\frac{\partial f_2}{\partial t} =\frac{c}{2}\frac{\partial f_2}{\partial u}+\frac{\sqrt{3}c}{2}\frac{\partial f_2}{\partial v} +\frac{\lambda(t)}{4}(f_0+f_1-4f_2),
\end{cases}
\end{equation}
with the condition that (use (\ref{probabilitaFacceOSM}))
\begin{equation}\label{condizioneFacciaTrasformata}
\int\int_{F_{ct}'} f(t,u,v)\dif u\dif v = P\big\{(X(t),Y(t),Z(t)\bigl)\in \mathring{F_{ct}^+}\big\} =\frac{e^{-\frac{\Lambda(t)}{2}}}{2}\Bigl(1-e^{-\frac{\Lambda(t)}{4}}\Bigr)^2.
\end{equation}
By means of the transformation $f_i(t,u,v) =\frac{e^{-\frac{\Lambda(t)}{2}}}{2} q_i(t,u,v)$ we obtain a new differential system,
\begin{equation}\label{sistemaInizialeFacciaTrasformata}
\begin{cases}
\frac{\partial q_0}{\partial t} = -c\frac{\partial q_0}{\partial u}+\frac{\lambda(t)}{4}(q_1+q_2-2q_0),\\
\frac{\partial q_1}{\partial t} = \frac{c}{2}\frac{\partial q_1}{\partial u}-\frac{\sqrt{3}c}{2}\frac{\partial q_1}{\partial v} +\frac{\lambda(t)}{4}(q_0+q_2-2q_1),\\
\frac{\partial q_2}{\partial t} =\frac{c}{2}\frac{\partial q_1}{\partial u}+\frac{\sqrt{3}c}{2}\frac{\partial q_1}{\partial v} +\frac{\lambda(t)}{4}(q_0+q_1- 2q_2),
\end{cases}
\end{equation}
where $q=q_0+q_1+q_2$ satisfies the boundary condition $\int\int_{F_{ct}'} q(t,u,v)\dif u\dif v =\Bigl(1-e^{-\frac{\Lambda(t)}{4}}\Bigr)^2$. This problem coincides with the differential system solved by the joint distributions of a symmetrically deviating planar random motion with directions $v_0,v_1,v_2$ and rate function $\lambda(t)/2, t>0$, see (\ref{sistemaDifferenzialeMotoPianoUniformeLO}) and Remark \ref{motoPianoTreDirezioniSD} (or equivalently a planar uniform motion with rate function $3\lambda(t)/4, t>0$). Therefore $q$ is the probability density of such a  motion on $\mathbb{R}^2$ (see Section \ref{sezioneMotoPianoTreDirezioni} for the constant rate case).

Note also that 
\[ P\{W(t) = ct\} = P\{X(t)+Y(t)+Z(t) = ct \} =   P\big\{(X(t),Y(t),Z(t)\bigl)\in {F_{ct}^+}\big\} =\frac{e^{-\frac{\Lambda(t)}{2}}}{2}.\]
Thus, $q$ is the conditional probability density of $\bigl(U(t),V(t)\bigr)$, given that $W(t) = ct$, and
\begin{align}
P\{&X(t)\in \dif x, Y(t)\in \dif y,X(t)+Y(t)+Z(t) = ct\}/(\dif x\dif y)\\
& = \frac{3\sqrt{3}}{2}f\Bigl(t,\frac{3x-ct}{2},\sqrt{3}y+\frac{x-3ct}{2}\Bigr) = \frac{3\sqrt{3}}{4}e^{-\frac{\Lambda(t)}{2}}q\Bigl(t,\frac{3x-ct}{2},\sqrt{3}y+\frac{x-3ct}{2}\Bigr) \nonumber.
\end{align}

The next theorem states what we proved above, also in the case of an uniform motion, where the proof works in the same way.

\begin{theorem}\label{teoremaDistribuzioneFaccia}
Let  $(X,Y,Z) = \big\{\bigl(X(t),Y(t),Z(t)\bigr)\big\}_{t\ge0}$ an orthogonal random motion in $\mathbb{R}^3$ with rate function $\lambda\in C^2\bigl((0,\infty),(0,\infty)\bigl)$ such that $\Lambda(t) = \int_0^t\lambda(s)\dif s<\infty, t>0$. Let $p(t,x,y)\dif x\dif y = P\{X(t)\in \dif x, Y(t)\in \dif y,X(t)+Y(t)+Z(t) = ct\}$, with $(x,y,ct-x-y)\in \mathring F_{ct}^+$, then
\begin{equation}\label{leggeFacciaEnunciato}
p(t,x,y)=\frac{3\sqrt{3}}{4}e^{-\frac{\Lambda(t)}{2}}q\Bigl(t,\frac{3x-ct}{2},\sqrt{3}y+\frac{x-3ct}{2}\Bigr)
\end{equation}
where
\item[$(i)$] if $(X,Y,Z)$ is an OSM, $q$ is the transition density of a symmetrically deviating planar random motion with directions (\ref{velocitaMotoPianoLO}) and rate function $\lambda/2$;
\item[$(ii)$] if $(X,Y,Z)$ is an OUM, $q$ is the transition density of a uniform planar random motion with directions (\ref{velocitaMotoPianoLO}) and rate function $\lambda/2$.
\end{theorem}

It is straightforward that, if $\lambda(t) = \lambda>0\ \forall \ t$, then function $q$ in (\ref{leggeFacciaEnunciato}) is given by formula (\ref{distribuzioneMotoPianoUniformeLO}) (or equivalently by (\ref{distribuzioneIntegraleMotoPianoUniformeLO})) with respectively $3\lambda/4$ and $\lambda/2$ replacing $\lambda$ in the cases $(i)$ and $(ii)$.

\begin{remark}[Governing differential equation]
For the sake of brevity we only consider the case of an OSM with constant rate function. Let $f_i(t,x,y)\dif x\dif y = P\{ X(t)\in\dif x, Y(t)\in \dif y, X(t)+Y(t)+Z(t)=ct,D(t)=d_i\},\ i=0,1,2$. Now, $p$ is the sum of the $f_i$ which satisfy the following differential system
\begin{equation}\label{sistemaInizialeMotoSuFaccia}
\begin{cases}
\frac{\partial f_0}{\partial t} = -c\frac{\partial f_0}{\partial x}+\frac{\lambda}{4}(f_1+f_2-4f_0),\\
\frac{\partial f_1}{\partial t} = -c\frac{\partial f_1}{\partial y}+\frac{\lambda}{4}(f_0+f_2-4f_1),\\
\frac{\partial f_2}{\partial t} =\frac{\lambda}{4}(f_0+f_1-4f_2),
\end{cases}
\end{equation}
This can be computed by proceeding as we explained for the distribution on an edge, see (\ref{sistemaNonDifferenzialeInizialeSpigolo}) and (\ref{sistemaInizialeSpigolo}).
By means of the transformation $g_0 = f_0+f_1,\ g_1=f_0-f_1$ ($f_2$ is kept unchanged) and some calculation, we obtain a differential system in $g_0$ and $f_2$ only,
\begin{equation}\label{sistemaSecondoMotoSuFaccia}
\begin{cases}
\frac{\partial^2 g_0}{\partial t^2}= -c^2 \frac{\partial^2 g_0}{\partial x\partial y} -c \Bigl(\frac{\partial }{\partial x}+\frac{\partial }{\partial y}\Bigr)\Bigl(\frac{\partial g_0}{\partial t} +\lambda g_0-\frac{\lambda}{4}f_2 \Bigr) +\frac{\lambda}{2}\Bigl(\frac{\partial f_2}{\partial t}-4\frac{\partial g_0}{\partial t}\Bigr)+\frac{5}{16}\lambda^2(2f_2-3g_0),\\
 \frac{\partial f_2}{\partial t} =\frac{\lambda}{4}(g_0-4f_2).
\end{cases}
\end{equation}
Let us write the first equation as $\frac{\partial^2 g_0}{\partial t^2} = A g_0+ Bf_2$, with $A, B$ being suitable differential operators. By using a second transformation, $p=g_0+f_2,\ w = g_0-f_2$, we readily arrive at the problem
\begin{equation}\label{sistemaTerzoMotoSuFaccia}
\begin{cases}
\frac{\partial^2 w}{\partial t^2} =  \Bigl(A + B +\frac{3\lambda}{4}\frac{\partial }{\partial t}\Bigr)\frac{p}{2}+\Bigl(A - B -\frac{5\lambda}{4}\frac{\partial }{\partial t}\Bigr)\frac{w}{2}, \\
\Bigl(\frac{\partial }{\partial t} +\frac{3}{4}\lambda \Bigr) p = \Bigl(\frac{\partial }{\partial t} +\frac{5}{4}\lambda \Bigr) w.
\end{cases}
\end{equation}
Finally, by deriving twice with respect to $t$ the second equation of (\ref{sistemaTerzoMotoSuFaccia}) and by suitably using both equations of the system (\ref{sistemaTerzoMotoSuFaccia}), we arrive at
$$\Bigl(\frac{\partial }{\partial t} +\frac{3}{4}\lambda \Bigr) \frac{\partial^2 p}{\partial t^2} =   \Bigl( \lambda A + A\frac{\partial }{\partial t} + \frac{\lambda}{4}B-\frac{\lambda}{4}\frac{\partial^2 }{\partial t^2}\Bigr) p$$
which, by expanding $A$ and $B$, reads
\begin{equation*}
\frac{\partial^3p }{\partial t^3}+3\lambda\frac{\partial^2p }{\partial t^2}+ \frac{45\lambda^2}{16}\frac{\partial p }{\partial t} +\frac{25\lambda^3}{32}p = -c \Bigl(\frac{\partial }{\partial x}+\frac{\partial }{\partial y}\Bigr)\Bigl(\frac{\partial^2 }{\partial t^2} +2\lambda\frac{\partial}{\partial t}+\frac{15}{16}\lambda^2 \Bigr) p -c^2\frac{\partial^2}{\partial x\partial y}\Bigl(\frac{\partial }{\partial t}+\lambda\Bigr)p.
\end{equation*}
The interested reader can compute the governing equation of an OUM by following substantially the strategy adopted above.
\hfill$\diamond$
\end{remark}

\section{Random times spent along the coordinated axes}

We study the distribution of the random times that the particle spends moving parallel to each axis. In particular, we are interested in the stochastic vector process given by $(T_x,T_y,T_z) = \big\{\bigl(T_x(t),T_y(t),T_z(t)\bigr)\big\}_{t\ge0}$, where, for $t\ge0$, $T_x(t), T_y(t)$ and $T_z(t)$ respectively denote the times spent up to time $t$ along the $x$-axis (meaning moving either with directions $d_0$ or $d_3$), the $y$-axis (meaning moving with directions $d_1,d_4$) and the $z$-axis (meaning moving upward with directions $d_2$ or downward with direction $d_5$). Note that these random times satisfy almost surely $T_x(t)+T_y(t)+T_z(t)=t\ \forall\ t$. Therefore, the knowledge of two of them, for instance $T_x$ and $T_y$, is sufficient to describe the triplet.

We begin by studying the distribution of the marginal process $T_z$. The next theorem states that, both in the case of an OSM and an OUM, $T_z$ coincides with a particular one-dimensional random motion that either moves with velocity $1$ or stops. Clearly, thanks to the symmetry properties of the orthogonal motions we are dealing with, the same result holds for the processes $T_x$ and $T_y$.

\begin{theorem}\label{teoremaTz}
Let $T_z = \{T_z(t)\}_{t\ge0}$ be the random time that an orthogonal motion in $\mathbb{R}^3$, $(X,Y,Z)$ with rate function $\lambda\in C^1\bigl((0,\infty),(0,\infty)\bigl) $, spends moving parallel to the $z$-axis. Then $T_z$ is an asymmetric one-dimensional random motion with velocity $V_z(t)\in\{0,1\} a.s., t>0$, such that
\begin{equation}\label{velocitaInizialeTz}
V_z(0) = \begin{cases}
1, \ \ \text{w.p. } 1/3,\\
0, \ \ \text{w.p. } 2/3. \end{cases} 
\end{equation}
Furthermore the changes of velocity follow the following rule:
\item[$(i)$] if $(X,Y,Z)$ is an OSM, when $V_z(t) = 1$ the rate function is $\lambda_1(t) = \lambda(t), t>0,$ and when $V_z(t) = 0$ the rate function is $\lambda_0(t) = \lambda(t)/2, t>0$;
\item[$(ii)$] if $(X,Y,Z)$ is an OUM, when $V_z(t) = 1$ the rate function is $\lambda_1(t) = 2\lambda(t)/3, t>0,$ and when $V_z(t) = 0$ the rate function is $\lambda_0(t) = \lambda(t)/3, t>0$.
\end{theorem}

Note that in the case of a rate function $\lambda$ such that $\Lambda(t) = \int_0^t \lambda(s)\dif s<\infty\ \forall \ t$, then for all $t>0$, $T_z(t)$ has two singular components at the border of its support, $[0,t]$:
\begin{itemize}
\item[($a$)] if $(X,Y,Z)$ is an OSM, $P\{T_z(t) = 0\} = 2e^{-\Lambda(t)/2}/3$ and $P\{T_z(t) = t\} = e^{-\Lambda(t)}/3$,
\item[($b$)] if $(X,Y,Z)$ is an OUM, $P\{T_z(t) = 0\} = 2e^{-\Lambda(t)/3}/3$ and $P\{T_z(t) = t\} = e^{-2\Lambda(t)/3}/3$.
\end{itemize}

\begin{proof}
Firstly, we observe that for $t\ge0$, $V_z(t) = 1$ if the orthogonal motion $(X,Y,Z)$ moves with directions $d_2$ or $d_5$ at time $t$, while $V_z(t) =0$ if $(X,Y,Z)$ moves parallel to the $(x,y)$-plane, that is with one of the four directions $d_0,d_1,d_3,d_4$. Then, at time $t=0$ we have (\ref{velocitaInizialeTz}). 

We now prove the theorem in the case of $(X,Y,Z)$ being an OSM (The OUM case is derived by proceeding in the same manner). Let $t>0$. For $s\in [0,t]$, we define $p(t,s) \dif s= P\{T_z(t)\in \dif s\}$ and $f_i(t,s) \dif s= P\{T_z(t)\in \dif s,V_z(t)=i\}$, with $i=0,1$. The functions $f_i$ satisfy
\begin{equation}\label{sistemaInizialeTz}
\begin{cases}
\frac{\partial f_0}{\partial t} = -\frac{\lambda(t)}{2}f_0+\lambda(t)f_1,\\
\frac{\partial f_1}{\partial t} = -\frac{\partial f_1}{\partial s}+\frac{\lambda(t)}{2}(f_0-2f_1).
\end{cases}
\end{equation}
For instance, the first equation of (\ref{sistemaInizialeTz}) is obtained by observing that 
$$ f_0(t+\dif t,s) = f_0(t,s)\bigl(1-\lambda(t)\dif t\bigr)+f_1(t,s-\dif t)\lambda(t)\dif t+f_0(t,s)\frac{\lambda(t)}{2}\dif t+o(\dif t),$$
where the first term in the right-hand side pertains to the fact that the particle moves on the $(x,y)$-plane and nothing changes its status. The second term is due to the fact that the particle moves vertically ($d_2$ or $d_5$) and the occurrence of a Poisson event takes it to the orthogonal plane. Finally, the last term is due to the fact that the Poisson event occurring while the particle moves on the $(x,y)$-plane, leaves it on the same plane because the vertical directions are not chosen. The second equation of (\ref{sistemaInizialeTz}) can be derived with similar arguments.

Now, by means of the change $p = f_0+f_1,\ w = f_0-f_1$ and with the same technique used above, we obtain a second-order partial differential equation on $p$, that is
\begin{equation}\label{equazioneLeggeTz}
\frac{\partial^2 p}{\partial t^2} +\frac{\partial^2 p}{\partial t \partial s} +\frac{3\lambda(t)}{2}\frac{\partial p}{\partial t}+\frac{\lambda(t)}{2}\frac{\partial p}{\partial s}= 0.
\end{equation}
The proof ends by observing that equation (\ref{equazioneLeggeTz}) coincides with equation (4.14) of Cinque and Orsingher \cite{CO2021a} with $c_1 = 1,c_2 = 0,\lambda_1(t) = \lambda(t)$ and $\lambda_2(t) = \lambda(t)/2$ (in the statement of the theorem $\lambda_0(t) = \lambda_2(t)$) and by keeping in mind the boundary conditions, $p\ge0$, $\int_0^t p(t,s)\dif s = 1-\frac{2}{3}e^{-\frac{\Lambda(t)}{2}}-\frac{1}{3}e^{-\Lambda(t)}$ (see above in ($a$)), if $\Lambda<\infty $, and $\int_0^t p(t,s)\dif s =1\ \forall\ t $ otherwise. 
\end{proof}

\begin{remark}
In the case of a constant rate function $\lambda(t) = \lambda> 0\ \forall \ t,$ we are able to display the explicit form of the transition density of $T_z$ by suitably adapting the known conditional distributions of the asymmetric telegraph process (see for instance formula (2.15) of Cinque \cite{C2022} or at the end of Section 2.2 of Lopez and Ratanov \cite{LR2014}). For instance, if $(X,Y,Z)$ is an OSM, with (\ref{velocitaInizialeTz}) at hand, we obtain that, for $s\in (0,t)$,
\begin{align}
P\{T_z(t)\in \dif s\}/\dif s & = \frac{\lambda e^{-\frac{\lambda}{2}(t+s)}}{3} \Bigl[2\,I_0\bigl(\lambda \sqrt{2s(t-s)}\bigr) + \frac{2t-s}{\sqrt{2s(t-s)}}\,I_1\bigl(\lambda \sqrt{2s(t-s)}\bigr)\Bigr]\nonumber \\
& = \frac{ e^{-\frac{\lambda}{2}(t+s)}}{3} \Bigl(2\lambda +3\frac{\partial }{\partial t}+2\frac{\partial }{\partial s}\Bigr)\,I_0\bigl(\lambda \sqrt{2s(t-s)}\bigr).\label{formaEsplicitaTempoTz}
\end{align}
To check that probability density (\ref{formaEsplicitaTempoTz}) satisfies the condition $\int_0^t P\{T_z(t)\in \dif s\}= 1-\frac{e^{-\lambda t}}{3} - \frac{2}{3}e^{-\frac{\lambda t}{2}}$, the following integral of the modified Bessel function of order $0$ is useful,
\begin{equation}\label{integraleEsponenzialePerBessel}
\int_0^t e^{\beta s} I_0\bigl(\alpha\sqrt{s(t-s)}\bigr) \dif s = \frac{e^{\frac{\beta t}{2}}}{\sqrt{\alpha^2 + \beta^2}} \Bigl(e^{\frac{t}{2}\sqrt{\alpha^2 + \beta^2}} -e^{-\frac{t}{2}\sqrt{\alpha^2 + \beta^2}} \Bigr),
\end{equation}
with $\alpha,\beta\in \mathbb{R}$.

Alternatively, one can achieve formula (\ref{formaEsplicitaTempoTz}) by considering that the function
\begin{align}
p(t,s) = e^{-\frac{\lambda}{2}(t+s)} \Bigl(A +B\frac{\partial }{\partial t}+C\frac{\partial }{\partial s}\Bigr)\,I_0\bigl(\lambda \sqrt{2s(t-s)}\bigr)
\end{align}
satisfies (\ref{equazioneLeggeTz}) (this can be proved by observing that each term involving the Bessel function satisfies the differential equation obtained by applying the transformation $p(t,s) = e^{-\frac{\lambda(t+s)}{2}}q(t,s)$). Then, the coefficients $A,B,C$ are calculated by taking into account the singular component of the motion, derived by suitably applying the integral (\ref{integraleEsponenzialePerBessel}).
 \hfill$\diamond$
\end{remark}

\subsection{Joint distribution of $\bigl(T_z(t), Z(t)\bigr)$}

\begin{figure}
		\centering
		\begin{tikzpicture}[scale = 0.71]
		\draw[dashed, gray] (0,0) -- (6,0) node[below right, black, scale = 0.9]{$t$};
		\draw[dashed, gray] (0,0) -- (6,4.8) circle(2pt) node[above, black, scale = 0.9]{A};
		\filldraw[blue] (6,4.8) circle(2pt);
		\draw[dashed, gray] (6,4.8) -- (6,-4.8) node[below, black, scale = 0.9]{B};
		\filldraw[blue] (6,-4.8) circle(2pt);		
		\draw[dashed, gray] (6,-4.8) -- (0,0) node[below left, black, scale = 0.9]{O};
		\draw (0,4.8) node[right, scale =0.9]{$ct$};
		\draw (0,-4.8) node[left, scale =0.9]{$-ct$};
		\draw[->, thick, gray] (-1.5,0) -- (7.5,0) node[below, scale = 1, black]{$\pmb{T_z(t)}$};
		\draw[->, thick, gray] (0,-5.5) -- (0,5.5) node[left, scale = 1, black]{ $\pmb{Z(t)}$};
		\draw (0,0)--(2,1.6)--(2.8,0.96)--(5,2.72)--(5, 2.72)--(5.5,3.12);
		\filldraw (0,0) circle (1pt); \filldraw (2,1.6) circle (1pt); \filldraw (2.8,0.96)circle (1pt); \filldraw (5, 2.72) circle (1.5pt); \filldraw (5.5,3.12) circle (1pt);
		\draw (0,0)--(1,-0.8)--(1,-0.8)--(2.5, -2)--(3.5,-1.2)--(3.7,-1.04)--(4,-1.28);
		\filldraw (0,0) circle (1pt); \filldraw (1,-0.8) circle (1.5pt); \filldraw(2.5,-2) circle (1pt); \filldraw(3.5,-1.2) circle (1pt); \filldraw(3.7,-1.04) circle (1pt); \filldraw(4,-1.28) circle (1pt);
		\draw (0,0)--(0.4,-0.32)--(1.8,0.8)--(2.5, 0.24)--(3, 0.64)--(3.3,0.88)--(4,1.44)--(5.2,0.48)--(5.7,0.08)--(6,0.32);
		\filldraw (0,0) circle (1pt); \filldraw (0.4,-0.32) circle (1pt); \filldraw(1.8,0.8) circle (1pt); \filldraw(2.5, 0.24) circle (1pt); \filldraw(3, 0.64) circle (1pt); \filldraw(3.3,0.88) circle (1pt); \filldraw(4,1.44) circle (1pt);\filldraw(5.2,0.48) circle (1pt); \filldraw(5.7,0.08) circle (1pt); \filldraw(6,0.32) circle (1pt);
		\end{tikzpicture}
		\caption{\small Sample paths of $(T_z, Z)$ in the case of an orthogonal uniform motion.}\label{graficoTzZ}
\end{figure}
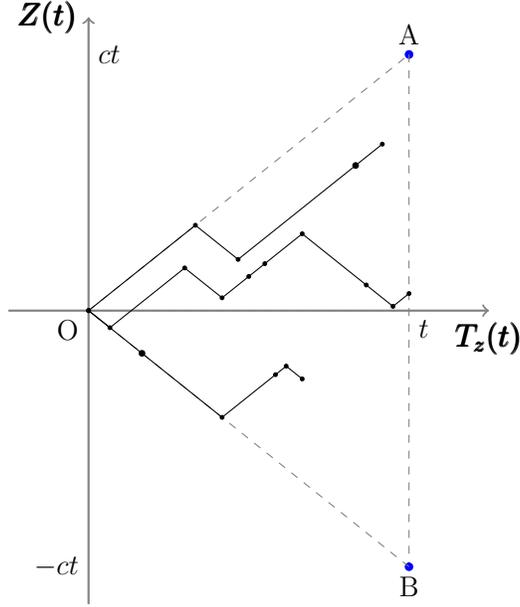

We present some considerations on the vector process $(T_z,Z)=\big\{\bigl(T_z(t), Z(t)\bigr)\big\}_{t\ge0}$ describing the time $T_z(t)$ that the particle spends parallel to the $z$-axis in the time interval $[0,t]$ and the vertical level $Z(t)$ reached at time $t$. Both in the OSM and the OUM case we observe that $T_z(t) = 0\implies Z(t) = 0$, $Z(t) =\pm ct \implies T_z(t) = t$ and $T_z(t) = s \implies Z(t)\in[-cs, cs]$ with $0<s<t$. However, for the OSM we also have that  $T_z(t) = t \implies Z(t) \in\{\pm ct\}$. Thus, the support of the vector, at time $t>0$ is 
$$ \mathcal{T}_{ct} = \{(s,z)\in\mathbb{R}^2\,:\,0\le s\le t,-cs\le z\le cs\}$$
as shown in Figure \ref{graficoTzZ}. We must keep in mind that in the case of an OSM the motion reaches with null probability the set $\{(s,z)\in\mathbb{R}^2\,:\,s=t,-ct<z<ct\}$, i.e. the inner points of the segment $AB$ in Figure \ref{graficoTzZ}.
\\

Now, we assume a constant rate function $\lambda(t)=\lambda>0  \ \forall \ t$.
By means of the above method we obtain that, if $(X,Y,Z)$ is an OUM, the distribution $p(t,s,z)\dif s\dif z = P\{T_z(t)\in \dif s, Z(t)\in\dif z\}, (s,z)\in \mathring{\mathcal{T}_{ct}}$, satisfies the following third-order differential equation:
\begin{equation}\label{equazioneTempoPosizioneCongiunta}
\frac{\partial^3 p}{\partial t^3} +2 \frac{\partial^3 p}{\partial t^2 \partial s} + \frac{\partial^3 p}{\partial t \partial s^2} +\frac{7\lambda}{3}\frac{\partial^2 p}{\partial t^2} +\frac{\lambda}{3}\frac{\partial^2 p}{\partial s^2}+\frac{7\lambda}{3} \frac{\partial^2 p}{\partial t \partial s} + \lambda^2\frac{\partial p}{\partial t} +\frac{\lambda^2}{3}\frac{\partial p}{\partial s}= c^2 \frac{\partial^3 p}{\partial t \partial z^2} + \frac{c^2\lambda}{3}\frac{\partial^2 p}{\partial z^2} .
\end{equation}
This is the equation governing the transition density of a planar random motion with directions $v_0 = (0,0), v_1=(1,c)$ and $v_2=(1,-c)$ whose changes of direction are paced by a homogeneous Poisson process with constant rate $\lambda$ and such that at each Poisson event (and at time $t=0$) the new direction is chosen with the following rule: $v_0$ with probability $2/3$, $v_1$ and $v_2$ with probability $1/6$ each.

Similar considerations hold in the case of an OSM, where the third-order differential equation governing the joint probability $p$ reads
\begin{equation}\label{equazioneTempoPosizioneCongiuntaOSM}
\frac{\partial^3 p}{\partial t^3} +2 \frac{\partial^3 p}{\partial t^2 \partial s} + \frac{\partial^3 p}{\partial t \partial s^2} +\frac{5\lambda}{2}\frac{\partial^2 p}{\partial t^2} +\frac{\lambda}{2}\frac{\partial^2 p}{\partial s^2}+3\lambda \frac{\partial^2 p}{\partial t \partial s} + \frac{3\lambda^2}{2}\frac{\partial p}{\partial t} +\frac{\lambda^2}{2}\frac{\partial p}{\partial s}= c^2 \frac{\partial^3 p}{\partial t \partial z^2} + \frac{c^2\lambda}{2}\frac{\partial^2 p}{\partial z^2} .
\end{equation}
Equation (\ref{equazioneTempoPosizioneCongiuntaOSM})  can be written down in the alternative form
\begin{equation}\label{equazioneTempoPosizioneCongiuntaOSM2}
\Bigl(\frac{\partial p}{\partial t} + \frac{\lambda}{2}\Bigr)\Biggl[\Bigl( \frac{\partial p}{\partial t} +\frac{\partial p}{\partial s}\Bigr)^2+2\lambda \Bigl( \frac{\partial p}{\partial t} +\frac{\partial p}{\partial s}\Bigr)-c^2\frac{\partial^2 p}{\partial z^2}\Biggr]p +\frac{\lambda^2}{2}\Bigl( \frac{\partial p}{\partial t} -\frac{\partial p}{\partial s}\Bigr)p=0,
\end{equation}
where the telegraph operator appears in the square brackets.

By means of the change of variables $z_1 = t-s,z_2 = cs+z,z_3 = cs-z$ equation (\ref{equazioneTempoPosizioneCongiuntaOSM2}) can be further developed as
\begin{equation}\label{equazioneTempoPosizioneCongiuntaOSMTrasformata}
\frac{\partial^3 p}{\partial z_1\partial z_2 \partial z_3} = \frac{\lambda^2}{8c} \Bigl( \frac{\partial p}{\partial z_2}+\frac{\partial }{\partial z_3}\Bigr)p.
\end{equation}
By applying the transformation $w = \sqrt[3]{z_1z_2z_3}$ the operator on the left-hand-side of (\ref{equazioneTempoPosizioneCongiuntaOSMTrasformata}) converts into the third-order Bessel operator, in detail $\frac{1}{3^3w^2}\frac{\partial }{\partial w}\Bigl(w\frac{\partial}{\partial w}\bigl(w\frac{\partial}{\partial w}\bigr)\Bigr)$.

Unfortunately, for both versions of the orthogonal motion we are not able to provide explicitly the transition density of $(T_z,Z)$. In particular, it is interesting to observe that, even if, intuitively, one may think that the conditional probability $P\{Z(t)\in \dif z\,|\,T_z(t)=s\},\ z\in(-cs,cs),$ coincides with the probability $P\{\mathcal{T}(s)\in \dif z\}$, with $\mathcal{T}$ being a symmetric telegraph process with velocity $c$ (and a constant rate function depending on $\lambda$), this fascinating relationship does not exist. Therefore,
\begin{align}\label{equazioneMancatoProdotto}
p(t,s,z)\dif s\dif z &= P\{T_z(t)\in \dif s, Z(t)\in\dif z\} = P\{T_z(t)\in \dif s\}P\{Z(t)\in \dif z\,|\,T_z(t) =s\} \\
& \neq P\{T_z(t)\in \dif s\}P\{\mathcal{T}(s)\in \dif z\} = g(t,s) l(s,z) \dif s \dif z
\end{align}
and this is because the time horizon $t$ cannot be negleted when considering the time reached at time $s$. In detail, for an OSM we obtain that the densitiy of the first member of (\ref{equazioneMancatoProdotto}) satisfies equation (\ref{equazioneTempoPosizioneCongiuntaOSMTrasformata}), while the last member satisfies a slightly different equation. This difference is equal to $-\lambda\frac{\partial g}{\partial t}\frac{\partial l}{\partial s}$, with $g$ satisfying (\ref{equazioneLeggeTz}) and $l$ satisfying the telegraph equation with rate $\lambda/2$.
\\

The following remarks contain the distribution of the joint vector $(T_z, Z)$ over the edges of its support. In the case of an OUM we are able to derive explicit results when $\lambda(t) = \lambda>0 \forall \ t$.

\begin{remark}
Let us consider the OUM. We study the particular case where $Z(t)=cT_z(t)$, meaning that the vector process $(X,Y,Z)$ never moved with direction $d_5 = (0,0,-1)$ (the case $Z(t)=-cT_z(t)$ is equivalent). In this case we can assume an integrable rate function $\lambda$ such that $\Lambda(t)=\int_0^t \lambda(s)\dif s<\infty, t>0$.
\\Let $p(t,s)\dif s = P\{T_z(t)\in \dif s, Z(t) = cs\}$, $f_0 (t,s)\dif s =P\big\{T_z(t)\in \dif s, Z(t) = cs, D(t) \in \{d_0,d_1,d_3,d_4\}\bigl\}$ and  $f_1 (t,s)\dif s =P\big\{T_z(t)\in \dif s, Z(t) = cs, D(t) = d_2 \bigl\}$, , with $s\in (0,t)$, then $p = f_0+f_1$. By proceeding as we showed above, we can check that the functions $f_i$ are such that 
$$\begin{cases}
\frac{\partial f_1}{\partial t} = -\frac{\partial f_1}{\partial s }+\frac{\lambda(t)}{6}(f_0-5f_1),\\ \frac{\partial f_0}{\partial t} = \frac{\lambda(t)}{3}\,(2f_1-f_0),
\end{cases}$$
 and by means of the change $p = f_0+f_1, w = f_0-f_1$ we obtain the second-order differential problem in $p$ only,
\begin{equation}\label{sistemaTzCongiuntoZ=cTz}
\begin{cases}
\frac{\partial^2 p}{\partial t^2}+\frac{\partial^2 p}{\partial t \partial s}+ \frac{7\lambda(t)}{6}\frac{\partial p}{\partial t}+\frac{\lambda(t)}{3}\frac{\partial p}{\partial s }+\frac{1}{6}\bigl(\lambda(t)^2 +\lambda'(t)\bigr)p = 0, \\
p\ge0,\ \ \int_0^t p(t,s)\dif s= \frac{5}{6}e^{-\frac{\Lambda(t)}{6}}-\frac{2}{3}e^{-\frac{\Lambda(t) }{3}} -\frac{1}{6}e^{-\frac{5\Lambda(t)}{6}}.
\end{cases}
\end{equation}
The second boundary condition follows by observing that 
\begin{align*}
\int_0^t &p(t,s)\dif s = P\{Z(t) = cT_z(t) \} -P\{T_z(t) = t, Z(t) = ct \}-P\{T_z(t) = 0, Z(t) = 0 \} \\
&=P\{Z(t) = cT_z(t) \} - P\big\{D(s) =d_2\text{ for }  s\in [0,t]\big\} - P\big\{D(s)\in \{d_0,d_1,d_3,d_4\} \text{ for }  s\in [0,t]\big\}.
\end{align*}
and $P(Z(t) = cT_z(t)) = P\big\{D(s)\not = d_5 \text{ for }  s\in [0,t]\big\}=\frac{5}{6}e^{-\frac{\Lambda(t)}{6}}$. Then, by using the change $p(t,s)=5/6\,e^{-\frac{\Lambda(t)}{6}}q(t,s)$, we have that $q(t,s)\dif s = P\{T_z(t) \in\dif s\,|\, Z(t)=cT_z(t) \}$ and system (\ref{sistemaTzCongiuntoZ=cTz}) transforms into
\begin{equation}\label{sistemaTzCondizionatoZ=cTz}
\begin{cases}
\frac{\partial^2 q}{\partial t^2}+\frac{\partial^2 q}{\partial t \partial s}+ \frac{5\lambda(t)}{6}\frac{\partial p}{\partial t}+\frac{\lambda(t)}{6}\frac{\partial p}{\partial s }= 0, \\
q\ge0,\ \ \int_0^t q(t,s)\dif s= 1-\frac{4}{5}e^{-\frac{\Lambda(t)}{6}} -\frac{1}{5}e^{-\frac{2\Lambda(t)}{3}}.
\end{cases}
\end{equation}
By keeping in mind the arguments applied in the proof of Theorem \ref{teoremaTz}, it is easy to see that the differential problem (\ref{sistemaTzCondizionatoZ=cTz}) is solved by the distribution of a one-dimensional telegraph process with velocities $c_1=1, c_2 = 0$, rate functions $\lambda_1(t) = 2\lambda(t)/3 ,\ \lambda_2(t) = \lambda(t)/6, t>0,$ and initial velocity $V(0) = 1$ with probability $1/5$ and $V(0) = 0$ with probability $4/5$. The distribution of the starting speed can be explained by considering that, conditioned on $Z(t) = cT_z(t)$, the particle chooses among five directions and only one concerns the $z$-axis ($d_2$).
\\In the case of a constant rate, the interested reader can compute the explicit distribution by proceeding as we showed for probability (\ref{formaEsplicitaTempoTz}). By using the explicit form, it is easy to see that
$$P\{T_z(t)\in\dif s, Z(t)=cs\}\not = P\{T_z(t)\dif s\} P\{\mathcal{T}(s) = cs\}$$
with $\mathcal{T}$ being a symmetric telegraph process with velocity $c$ (and a constant rate function depending on $\lambda$).
\\

We point out that in the case of an OSM the reader can proceed as above, but the probability $ P\{Z(t) = cT_z(t) \}$ has a complicated form. In particular, the evaluation of this probability poses serious combinatorial problems. For example, if at time $t=0$ the initial direction is $d_2$ (upward motion), $N(t) = n$ and the number of vertical displacements (excluding the first one, which occurs at time $t=0$) is $k\ge1 $, we have two types of sequences (let $\hat d_0$ denotes a speed in the set of the horizontal directions $\{d_0,d_1,d_3,d_4\}$) which both starts with $\hat d_0$ since after a horizontal step follows a vertical movement with probability 1,
\begin{itemize}
\item $(d_2)\ \hat d_0, \dots, d_2,\hat d_0,\dots, \hat d_0,d_2$, where the last displacement is performed with direction $d_2$,
\item $(d_2)\ \hat d_0,\dots, d_2,\hat d_0, \dots, \hat d_0$, where the last displacement is performed with direction $\hat d_0$.
\end{itemize}
The number of possible sequences is respctively $\binom{n-k-1}{k-1}$ and $\binom{n-k-1}{k}$. The probability of these types of runs (by also keeping in mind the starting speed $d_2$) is equal to, $n\in \mathbb{N}$ and integer $0\le k\le n/2$,
\begin{equation}\label{probCombinatoria2}
 P\{Z(t) = cT_z(t),\, N_2(t) = k+1\,|\,N(t) = n, D(0)= d_2 \} = \frac{1}{6}\Biggl[\frac{1}{2^n}\binom{n-k-1}{k-1} + \frac{1}{2^{n-1}}\binom{n-k-1}{k} \Biggr],
\end{equation}
where $N_2(t)$ denotes the number of displacements with direction $d_2$ in the time interval $[0,t]$ (since we are assuming $D(0) = d_2$ it is clear that $N_2(t)\ge1\ a.s.$).
\\In the calculation of (\ref{probCombinatoria2}) one must consider that switches $\hat d_0 \longrightarrow d_2$ occurs with probability $1/2$ ($k$ times in both case), from $d_2 \longrightarrow \hat d_0$ with probability $1$ ($k-1$ times in the first case and $k$ in the second one) and $\hat d_0 \longrightarrow\hat d_0$ with probability $1/2$ ($n-2k$ times in the first case and $n-2k-1$ in the second one).
\\

By means of similar arguments, one can show that, if at time $t=0$ the initial direction is $\hat d_0$ (meaning a horizontal direction), then the following probability holds for $n\in\mathbb{N}$ and integer $1\le k\le n/2+1$ (by considering $\binom{n}{k} = 0$ if $n<k$), 
\begin{align}
 P&\{Z(t) = cT_z(t),\, N_2(t) = k\,|\,N(t) = n, D(0)= d_0 \} \nonumber\\
& = \frac{2}{3}\Biggl[\frac{1}{2}\frac{1}{2^n}\binom{n-k-1}{k-1} + \frac{1}{2}\frac{1}{2^{n-1}}\binom{n-k-1}{k}  +\frac{1}{4}\frac{1}{2^{n-1}}\binom{n-k-1}{k-2} + \frac{1}{4}\frac{1}{2^{n-2}}\binom{n-k-1}{k-1} \Biggr]\label{probCombinatoria0}\\
& = \frac{1}{3}\Biggl[\frac{1}{2^{n-1}}\binom{n-k}{k} + \frac{1}{2^{n}}\binom{n-k}{k-1} \Biggr].\nonumber
\end{align}
The four terms in (\ref{probCombinatoria0}) concern the following four possible sequences:
\begin{itemize}
\item starting with $\hat d_0$: $(\hat d_0)\, \hat d_0, \dots, d_2,\hat d_0,\dots, \hat d_0,d_2$ and $(\hat d_0)\, \hat d_0,\dots, d_2,\hat d_0, \dots, \hat d_0$, 
\item starting with $d_2$: $(\hat d_0)\, \hat d_2, \dots, d_2,\hat d_0,\dots, \hat d_0,d_2$ and $(\hat d_0)\,\hat d_2,\dots, d_2,\hat d_0, \dots, \hat d_0$.\hfill$\diamond$
\end{itemize}
\end{remark}

\begin{remark}
In the case of an OUM if $T_z(t) = t$ then $Z(t)\in [-ct,ct]\ a.s.$ and we can show that its conditional probability is equal to the distribution of a symmetric one-dimensional telegraph process.
Let us consider $p(t,z)\dif z =P\{T_z(t) = t,Z(t) \in \dif z\} =\sum_{i\in \{2,5\}} P\{T_z(t) = t, Z(t)\in \dif z,D(t) = d_i\} =  \sum_{i\in \{2,5\}} f_i(t,z)\dif z$, then $f_2,f_5$ are such that $\frac{\partial f_2}{\partial t} = -c\frac{\partial f_2}{\partial z} +\lambda(t)(f_5-5f_2)/6,\ \frac{\partial f_5}{\partial t} = c\frac{\partial f_5}{\partial z} +\lambda(t)(f_5-5f_2)/6$. By means of the change $p = f_2+f_5,w = f_2-f_5$ we obtain that $p$ satisfies
\begin{equation}\label{sistemaTzCondizionatoTz=t}
\begin{cases}
\frac{\partial^2 p}{\partial t^2}+\frac{5\lambda(t)}{3}\frac{\partial p}{\partial t }+ \frac{2}{3}\bigl(\lambda(t)^2+\lambda'(t)\bigr)p = c^2\frac{\partial^2 p}{\partial z^2 }, \\
p\ge0,\ \ \int_{-ct}^{ct} p(t,z)\dif z = P\{T_z(t) = t \} -P\big\{T_z(t) = t, Z(t) \in\{-ct, ct\} \big\} = \frac{1}{3}e^{-\frac{2\Lambda(t)}{3}}-\frac{1}{3}e^{-\frac{5\Lambda(t)}{6}}.
\end{cases}
\end{equation}
and with the transformation $p(t,z) = \frac{1}{3}e^{-2\Lambda(t)/3}\,
q(t,z)$ we obtain that $q(t,z)\dif z = P\{Z(t)\in \dif z\,|\,T_z(t) = t\}$ and by studying its governing differential system we obtain that it coincides with the absolutely continuous component of a symmetric telegraph process with rate function $\lambda(t)/6 \ \forall \ t$ and velocities $c$.
\hfill$\diamond$
\end{remark}

\subsection{Joint distribution of the vector process $(T_x, T_y, T_z)$}

The next theorem concerns the distribution of the vector process $(T_x,T_y,T_z)$, which satisfies $T_x(t)+ T_y(t)+T_z(t) = t, t\ge0$, and thus we only need to study the distribution of a couple.
\begin{theorem}\label{teoremaTriplettaTempi}
Let  $(T_x,T_y,T_z) = \big\{\bigl(T_x(t),T_y(t),T_z(t)\bigr)\big\}_{t\ge0}$ be the stochastic process describing the time spent on each coordinate axes by the process $(X,Y,Z) = \big\{\bigl(X(t),Y(t),Z(t)\bigr)\big\}_{t\ge0}$, an orthogonal random motion in $\mathbb{R}^3$ with rate function $\lambda\in C^2\bigl((0,\infty),(0,\infty)\bigl)$. Let $p(t,s,r)\dif s\dif r = P\{T_x(t)\in \dif s, T_y(t)\in \dif r,T_x(t)+T_y(t)+T_z(t) = t\}$, with $s,r>0, s+r<t$, then
\begin{equation}\label{leggeTriplettaTempi}
p(t,s,r)=\frac{3\sqrt{3}\,c^2}{2}q\Bigl(t,\frac{c}{2}(3s-t),\frac{\sqrt{3}c}{2}(s+2r-t)\Bigr),
\end{equation}
where
\item[$(i)$] if $(X,Y,Z)$ is an OSM, $q$ is the transition density of a symmetrically deviating planar random motion with directions (\ref{velocitaMotoPianoLO}) and rate function $\lambda$;
\item[$(ii)$] if $(X,Y,Z)$ is an OUM, $q$ is the transition density of a uniform planar random motion with directions (\ref{velocitaMotoPianoLO}) and rate function $\lambda$.
\end{theorem}

\begin{proof}
The process $(T_x,T_y)$ moves with unitary velocity along the directions $e_0 = (0,0), e_1 = (1,0), e_2 = (0,1)$. The proof works along the same lines of Subsection \ref{sottosezioneDistribuzioneFaccia} in order to prove Theorem \ref{teoremaDistribuzioneFaccia}. The system governing the joint distributions of the process and the current direction can be transformed into the corresponding system of a planar random motion with directions (\ref{velocitaMotoPianoLO}), that in the case of an OUM is given in (\ref{sistemaDifferenzialeMotoPianoUniformeLO}), by appling the transformation
\begin{equation}\label{trasformazioneTempiAleatori}
\begin{cases}
U(t) = \frac{c}{2}\bigl(3T_x(t)-t\bigr),\\
V(t) = \frac{\sqrt{3}c}{2}\bigl(T_x(t)+2T_y(t)-t\bigr),\\
\end{cases}
\begin{cases}
T_x(t) = \frac{1}{3c}\bigl(ct+2U(t)\bigr),\\
T_y(t)= \frac{1}{3c}\bigl(ct-U(t)+\sqrt{3}V(t)\bigr).\\
\end{cases}
\end{equation}
\end{proof}
Clearly, in the case that the rate function is such that $\Lambda(t) = \int_0^t\lambda(s)\dif s<\infty, t>0$, then also the distribution on the border of the support $ \{(s,r)\in \mathbb{R}^2\,:\, r,s\ge0,r+s\le t\}$ can be suitably obtained. Furthermore, if $\lambda(t)=\lambda >0 \ \forall \ t$, then the explicit form of the transition density $q$ in (\ref{leggeTriplettaTempi}) can be calculated from either (\ref{distribuzioneMotoPianoUniformeLO}) or (\ref{distribuzioneIntegraleMotoPianoUniformeLO}).

\section{Orthogonal motions in $\mathbb{R}^3$: absolutely continuous component}

Let $t\ge0$. In this section we examine some different methods to deal with the distribution of the random vector process $\bigl(X(t),Y(t),Z(t)\bigr)$ inside of the octahedron $S_{ct}$ in (\ref{ottaedro}). This task is by far more complicated than the study of planar random motions, since the third dimension (and thus the six possible orthogonal directions) implies the involvement of sixth-order equations. For $(x,y,z)\in \mathring{S_{ct}}$, put
$$ p(t,x,y,z)\dif x \dif y \dif z = P\big\{ X(t)\in \dif x, Y(t)\in \dif y, Z(t)\in \dif z\big\}. $$

\begin{theorem}
Let $(X,Y,Z) = \big\{\bigl(X(t),Y(t),Z(t)\bigr)\big\}_{t\ge0}$ be a standard orthogonal random motion in $\mathbb{R}^3$ with rate function $\lambda(t) = \lambda>0, \ \forall \ t.$ The transition density $p$ of  the random vector $(X,Y,Z)$ satisfies the sixth-order equation
\begin{align}\label{equazioneDifferenzialeSestoOrdine}
\Biggl[ \Bigl(&\frac{\partial}{\partial t} +\lambda\Bigr)^6 -\frac{3\lambda^2}{4}\Bigl(\frac{\partial}{\partial t} +\lambda\Bigr)^4 -\frac{\lambda^3}{4}\Bigl(\frac{\partial}{\partial t} +\lambda\Bigr)^3 \Biggr] p= c^6\frac{\partial^6 p}{\partial x^2\partial y^2\partial z^2}\\
& - c^4\Bigl(\frac{\partial}{\partial t} +\lambda\Bigr)^2\Bigl(\frac{\partial^4 }{\partial x^2\partial y^2}+\frac{\partial^4 }{\partial x^2\partial z^2} +\frac{\partial^4 }{\partial y^2\partial z^2}\Bigr) p+c^2\Biggl[\Bigl(\frac{\partial}{\partial t} +\lambda\Bigr)^4 -\frac{\lambda^2}{4}\Bigl(\frac{\partial}{\partial t} +\lambda\Bigr)^2 \Biggr] \Delta p,\nonumber
\end{align}
where $\Delta = \frac{\partial^2 }{\partial x^2}+\frac{\partial^2 }{\partial y^2}+\frac{\partial^2 }{\partial z^2}$ is the Laplace operator.
\end{theorem}

\begin{proof}
In order to derive equation (\ref{equazioneDifferenzialeSestoOrdine}) we proceed as follows.
Let $f_i(t,x,y,z)\dif x \dif y\dif z = P\{ X(t)\in\dif x, Y(t)\in \dif y,Z(t)\in \dif z,D(t)=d_i\},\ i=0,1,2,3,4,5$. Now, $p$ is the sum of the probability densities $f_i$ and which satisfy the following differential system
\begin{equation}\label{sistemaInizialeMotoSuFaccia}
\begin{cases}
\frac{\partial f_0}{\partial t} = -c\frac{\partial f_0}{\partial x}+\frac{\lambda}{4}(f_1+f_2+f_4+f_5-4f_0),\ \ \ \frac{\partial f_3}{\partial t} = c\frac{\partial f_3}{\partial x}+\frac{\lambda}{4}(f_1+f_2+f_4+f_5-4f_3),\\
\frac{\partial f_1}{\partial t} = -c\frac{\partial f_1}{\partial y}+\frac{\lambda}{4}(f_0+f_2+f_3+f_5-4f_1),\ \ \ \frac{\partial f_4}{\partial t} = c\frac{\partial f_4}{\partial y}+\frac{\lambda}{4}(f_0+f_2+f_3+f_5-4f_4),\\
\frac{\partial f_2}{\partial t} =-c\frac{\partial f_2}{\partial z}+\frac{\lambda}{4}(f_0+f_1+f_3+f_4-4f_2),\ \ \ \frac{\partial f_5}{\partial t} =c\frac{\partial f_5}{\partial z}+\frac{\lambda}{4}(f_0+f_1+f_3+f_4-4f_5).
\end{cases}
\end{equation}
Note that the left-hand equations of (\ref{sistemaInizialeMotoSuFaccia}) refer to the displacements with positive directions, while the right-hand ones correspond to the negatively-oriented movements. 
\\The pairwise coincidence of the sums of functions $f_i$ is due to the orthogonality and symmetry of motion. Now, by means of the transformation $g_0 = f_0+f_3, h_0=f_0-f_3,\ g_1 = f_1+f_4, h_1 = f_1-f_4,g_2=f_2+f_5,\ h_2=f_2-f_5$, system (\ref{sistemaInizialeMotoSuFaccia}) can be converted into a second-order differential system for the functions $g_0,g_1,g_2$ only,
\begin{equation}
\begin{cases}
\frac{\partial^2 g_0}{\partial t^2} = c^2\frac{\partial^2 g_0}{\partial x^2}-2\lambda\frac{\partial g_0}{\partial t} + \frac{\lambda^2}{2}(-2g_0+g_1+g_2)+ \frac{\lambda}{2}\frac{\partial }{\partial t}(g_1+g_2),\\
\frac{\partial^2 g_1}{\partial t^2} = c^2\frac{\partial^2 g_1}{\partial y^2}-2\lambda\frac{\partial g_1}{\partial t} + \frac{\lambda^2}{2}(g_0-2g_1+g_2)+ \frac{\lambda}{2}\frac{\partial }{\partial t}(g_0+g_2),\\
\frac{\partial^2 g_2}{\partial t^2} = c^2\frac{\partial^2 g_2}{\partial z^2}-2\lambda\frac{\partial g_2}{\partial t} + \frac{\lambda^2}{2}(g_0+g_1-2g_2)+ \frac{\lambda}{2}\frac{\partial }{\partial t}(g_0+g_1).
\end{cases}
\end{equation}
We apply now a second transformation, $p = g_0+g_1+g_2,\ w_1 = g_0+g_1-g_2,\ w_2 = g_0-g_1+g_2$, and some calculation lead to
\begin{equation}\label{sistemaPWEquazioneSetsoOrdine}
\begin{cases}
\frac{\partial^2 p}{\partial t^2} = \frac{c^2}{2}\Bigl(\frac{\partial^2 }{\partial x^2}+\frac{\partial^2 }{\partial y^2}\Bigr)p-\lambda\frac{\partial p}{\partial t} + \frac{c^2}{2}\Bigl(\frac{\partial^2 }{\partial x^2}(w_1+w_2)-\frac{\partial^2 w_2}{\partial y^2}-\frac{\partial^2 w_1}{\partial z^2}\Bigr),\\
\frac{\partial^2 w_1}{\partial t^2} = \frac{c^2}{2}\Bigl(\frac{\partial^2 }{\partial y^2}-\frac{\partial^2 }{\partial z^2}\Bigr)p+\Bigl(\frac{\lambda^2}{2}+\frac{\lambda}{2}\frac{\partial}{\partial t}\Bigr)p + \frac{c^2}{2}\Bigl(\frac{\partial^2 }{\partial x^2}(w_1+w_2)-\frac{\partial^2 w_2}{\partial y^2}+\frac{\partial^2 w_1}{\partial z^2}\Bigr) -\frac{5\lambda}{2}\frac{\partial w_1}{\partial t}-\frac{3\lambda^2}{2}w_1,\\
\frac{\partial^2 w_1}{\partial t^2} = -\frac{c^2}{2}\Bigl(\frac{\partial^2 }{\partial y^2}-\frac{\partial^2 }{\partial z^2}\Bigr)p+\Bigl(\frac{\lambda^2}{2}+\frac{\lambda}{2}\frac{\partial}{\partial t}\Bigr)p + \frac{c^2}{2}\Bigl(\frac{\partial^2 }{\partial x^2}(w_1+w_2)+\frac{\partial^2 w_2}{\partial y^2}-\frac{\partial^2 w_1}{\partial z^2}\Bigr) -\frac{5\lambda}{2}\frac{\partial w_2}{\partial t}-\frac{3\lambda^2}{2}w_2.
\end{cases}
\end{equation}
Now, by deriving twice with respect to $t$ the first equation of (\ref{sistemaPWEquazioneSetsoOrdine}) and by suitably using the other two equations, with some effort, we arrive at the following fourth-order differential equation
\begin{align}\label{equazionePWQuartoOrdine}
\frac{\partial^4 p}{\partial t^4} &= -\frac{c^4}{2}\Bigl(\frac{\partial^4 }{\partial x^2\partial y^2}+\frac{\partial^4 }{\partial x^2\partial z^2}+2\frac{\partial^4 }{\partial y^2\partial z^2}\Bigr)p + c^2\Bigl(\frac{\partial^2}{\partial t^2}+ \frac{3\lambda}{2}\frac{\partial}{\partial t}+\frac{\lambda^2}{2} \Bigr)\Delta p\nonumber\\ 
&\ \ \ -\Bigl(\frac{\partial^2}{\partial t^2}+\lambda\frac{\partial}{\partial t}\Bigr)\Bigl(\frac{5\lambda}{2}\frac{\partial}{\partial t}+\frac{3}{2}\lambda^2 \Bigr)p-\lambda \frac{\partial^3p}{\partial t^3}+\frac{c^4}{2}\Bigl(\frac{\partial^4 }{\partial y^2\partial z^2}(w_1+w_2)-\frac{\partial^4 w_1}{\partial x^2\partial y^2}-\frac{\partial^4 w_2}{\partial x^2\partial z^2}\Bigr),
\end{align}
where still the auxiliary functions $w_1,w_2$ appear.

From the second and third equations of (\ref{sistemaPWEquazioneSetsoOrdine}) we obtain 
\begin{align}\label{sistemaSoliWEquazioneSetsoOrdine}
&\frac{c^4}{2}\frac{\partial^4 }{\partial y^2\partial z^2} \frac{\partial^2}{\partial t^2}(w_1+w_2) = \frac{c^4}{2}\frac{\partial^4 }{\partial y^2\partial z^2}\Biggl[\Bigl(c^2\frac{\partial^2 }{\partial x^2}-\frac{5\lambda}{2}\frac{\partial }{\partial t}-\frac{3\lambda^2}{2}\Bigr)(w_1+w_2) + \Bigl(\lambda^2+\lambda\frac{\partial}{\partial t}\Bigr)p \Biggr],\\
&-\frac{c^4}{2}\frac{\partial^4 }{\partial x^2\partial y^2} \frac{\partial^2}{\partial t^2}w_1 = \frac{c^4}{2}\frac{\partial^4 }{\partial x^2\partial y^2}\Biggl[-\frac{c^2}{2}\Bigl(\frac{\partial^2 }{\partial y^2}-\frac{\partial^2 }{\partial z^2}\Bigr)p -  \frac{c^2}{2}\Bigl(\frac{\partial^2 }{\partial x^2}(w_1+w_2)-\frac{\partial^2 w_2}{\partial y^2} +\frac{\partial^2 w_1}{\partial z^2}\Bigr)+ \nonumber\\
&\hspace{3cm}+ \frac{5\lambda}{2}\frac{\partial w_1}{\partial t}+\frac{3\lambda^2}{2}w_1- \Bigl(\frac{\lambda^2}{2}+\frac{\lambda}{2}\frac{\partial}{\partial t}\Bigr)p \Biggr],\nonumber\\
&-\frac{c^4}{2}\frac{\partial^4 }{\partial x^2\partial z^2} \frac{\partial^2}{\partial t^2}w_2 = \frac{c^4}{2}\frac{\partial^4 }{\partial x^2\partial z^2}\Biggl[\frac{c^2}{2}\Bigl(\frac{\partial^2 }{\partial y^2}-\frac{\partial^2 }{\partial z^2}\Bigr)p -  \frac{c^2}{2}\Bigl(\frac{\partial^2 }{\partial x^2}(w_1+w_2)+\frac{\partial^2 w_2}{\partial y^2} -\frac{\partial^2 w_1}{\partial z^2}\Bigr) \nonumber\\
&\hspace{3cm}+ \frac{5\lambda}{2}\frac{\partial w_2}{\partial t}+\frac{3\lambda^2}{2}w_2- \Bigl(\frac{\lambda^2}{2}+\frac{\lambda}{2}\frac{\partial}{\partial t}\Bigr)p \Biggr].\nonumber
\end{align}
We now explain how to express the derivatives of $w_1$ and $w_2$ in terms of derivatives of $p$ only. By summing up equations in (\ref{sistemaSoliWEquazioneSetsoOrdine}) and subsequently deriving twice with respect to $t$, we obtain a differential equation involving $p, w_1$ and $w_2$. The terms concerning the functions $w_1$ and $w_2$ are the following ones
\begin{align}
&A_1 = -\frac{c^4}{2^2}\Bigl(5\frac{\partial}{\partial t} + 3\lambda\Bigr)\Bigl(\frac{\partial^4}{\partial y^2\partial z^2}(w_1+w_2) -\frac{\partial^4 w_1}{\partial x^2 \partial y^2}-\frac{\partial^4 w_2}{\partial x^2 \partial z^2}\Bigr)\label{primoTermineDaEsplicitare}\\
&A_2 = \frac{c^6}{2^2}\Bigl(\frac{\partial^6}{\partial x^2\partial y^2\partial z^2}(w_1+w_2) -\frac{\partial^6}{\partial x^4\partial z^2}(w_1+w_2)+ \frac{\partial^6 w_1}{\partial x^4\partial z^2} -\frac{\partial^6}{\partial x^4\partial y^2}(w_1+w_2) +\frac{\partial^ 6w_2}{\partial x^4\partial y^2}\Bigr).\label{secondoTermineDaEsplicitare}
\end{align}
The term $A_1$ in (\ref{primoTermineDaEsplicitare}) can be easily expressed in terms of derivatives of function $p$ only by means of equation (\ref{equazionePWQuartoOrdine}). On the other hand, the term $A_2$ in (\ref{secondoTermineDaEsplicitare}) can be written as
$$ A_2 = -\frac{c^6}{2^2}\Bigl(\frac{\partial^4}{\partial x^2\partial y^2}+\frac{\partial^4}{\partial x^2\partial z^2}\Bigr)\Bigl(\frac{\partial^2}{\partial x^2}(w_1+w_2) -\frac{\partial^2w_2}{\partial y^2}-\frac{\partial^2w_1}{\partial z^2}\Bigr) $$
which easily turns into an expression of $p$ only by using the first equation of system (\ref{sistemaPWEquazioneSetsoOrdine}).

Now, with the above observations at hand, with suitable manipulations, we obtain the useful relationship
\begin{align}\label{ultimoPassagioW}
\frac{c^4}{2}&\Bigl(\frac{\partial^6 }{\partial t^2\partial y^2\partial z^2}(w_1+w_2)-\frac{\partial^6 w_1}{\partial t^2\partial x^2\partial y^2}-\frac{\partial^6 w_2}{\partial t^2\partial x^2\partial z^2}\Bigr) \\
& = c^6\frac{\partial^6 p}{\partial x^2\partial y^2\partial z^2} -\frac{c^4}{2} \Bigl(\frac{\partial^2 }{\partial t^2} +4\lambda \frac{\partial }{\partial t}+2\lambda^2\Bigr) \Bigl(\frac{\partial^4 }{\partial x^2\partial y^2}+\frac{\partial^4 }{\partial x^2\partial z^2} +\frac{\partial^4 }{\partial y^2\partial z^2}\Bigr) p +  \frac{c^4}{2}\frac{\partial^6 p}{\partial t^2\partial y^2\partial z^2}\nonumber\\
&\ \ \ - \Bigl(\frac{5\lambda}{2}\frac{\partial}{\partial t}+\frac{3}{2}\lambda^2 \Bigr)\Biggl[ \frac{\partial^4 }{\partial t^4} + \lambda\frac{\partial^3 }{\partial t^3} -\frac{c^2}{2} \Bigl(2\frac{\partial^2 }{\partial t^2} +3\lambda \frac{\partial }{\partial t}+\lambda^2\Bigr) \Delta + \Bigl(\frac{\partial^2 }{\partial t^2} +\lambda \frac{\partial }{\partial t}\Bigr)\Bigl(\frac{5\lambda}{2}\frac{\partial}{\partial t}+\frac{3}{2}\lambda^2 \Bigr) \Biggr]p.\nonumber
\end{align}
Finally, by deriving (\ref{equazionePWQuartoOrdine}) twice with respect to $t$ and by using (\ref{ultimoPassagioW}) we achieve the sixth-order differential equation (\ref{equazioneDifferenzialeSestoOrdine}).
\end{proof}

Note that under the so called Kac's conditions, $\lambda, c\longrightarrow\infty$ such that $\lambda/c^2\longrightarrow 1$, equation (\ref{equazioneDifferenzialeSestoOrdine}) converges to the space heat equation $\frac{\partial p}{\partial t} = \frac{\Delta p}{3}$ and the OSM converges in distribution to a three-dimensional Brownian motion (this can be equivalently observed in the case of an OUM).

\begin{remark}
Equation (\ref{equazioneDifferenzialeSestoOrdine}) can be substantially simplified by means of the transformation $p(t,x,y,z) = e^{-\lambda t}q(t,x,y,z)$. We obtain that $q$ satisfies
\begin{align}\label{equazioneDifferenzialeSestoOrdineTrasformata}
\frac{3\lambda^2}{4}\frac{\partial^2 }{\partial t^2} \Bigl(\frac{\partial^2 }{\partial t^2} +\frac{\lambda}{3}\frac{\partial }{\partial t} - \frac{c^2}{3}\Delta\Bigr)q= \Bigl(\frac{\partial^2}{\partial t^2} -c^2\frac{\partial^2 }{\partial x^2}\Bigr)\Bigl(\frac{\partial^2}{\partial t^2} -c^2\frac{\partial^2 }{\partial y^2}\Bigr)\Bigl(\frac{\partial^2}{\partial t^2} -c^2\frac{\partial^2 }{\partial z^2}\Bigr)q.
\end{align}
where in the right-hand side the product of D'Alembert operators appears.
\\

It is of interest to compare the above results with those of the planar case. The equation governing the transition probability $p(t,x,y)$ of an orthogonal standard motion with constant rate function $\lambda(t) = \lambda>0, t>0$, satisfies the fourth-order differential equation
\begin{equation}\label{equazioneMotoOrtogonalePiano}
\Bigl(\frac{\partial }{\partial t}+\lambda^2\Bigr) \Biggl[ \frac{\partial^2 }{\partial t^2} +2\lambda \frac{\partial }{\partial t} -c^2 \Bigl(\frac{\partial^2 }{\partial x^2} +\frac{\partial^2 }{\partial y^2} \Bigr)\Biggr]p +c^4 \frac{\partial^4p}{\partial x^2\partial y^2}=0.
\end{equation}\label{equazioneMotoPianoTrasformazione}
By means of the exponential transformation $p(t,x,y) = e^{-\lambda t}q(t,x,y)$ equation (\ref{equazioneMotoOrtogonalePiano}) reduces to
\begin{equation}
 \Bigl(\frac{\partial^2}{\partial t^2} -c^2\frac{\partial^2 }{\partial x^2}\Bigr)\Bigl(\frac{\partial^2}{\partial t^2} -c^2\frac{\partial^2 }{\partial y^2}\Bigr)q = \lambda^2\frac{\partial^2 q}{\partial t^2}
\end{equation}
By comparing (\ref{equazioneMotoPianoTrasformazione}) with (\ref{equazioneDifferenzialeSestoOrdineTrasformata}) we note that the increase of dimension, or more precisely the assumption that the motion can move along two additional directions in the third dimension, implies a further D'Alembert operator (involving the coordinate of the third dimension, $z$) while the time derivatives involve a telegraph-type operator.\hfill$\diamond$
\end{remark}

The joint probability distribution of both the OSM and the OUM can be expressed by means of the following integral representation,
\begin{align}\label{motoOrtogonaleComponenteAssolutamenteContinua}
p(t,x,&y,z) \dif x\dif y\dif z= \int_{\frac{|x|}{c}}^{t-\frac{|y|+|z|}{c}} \int_{\frac{|y|}{c}}^{t-t_x - \frac{|z|}{c}} P\{T_x\in \dif t_x, T_y(t)\in \dif t_y\}\\
&\times P\{X(t)\in \dif x\,|\,T_x(t)=t_x\}P\{Y(t)\in \dif y\,|\, T_y(t)=t_y\}P\{Z(t)\in \dif z\,|\, T_z(t)=t-t_x-t_y\}.\nonumber
\end{align}
We point out that the first factor in the integral is treated in Theorem \ref{teoremaTriplettaTempi}. It is important to observe that we can split the joint probability of $X(t), Y(t),Z(t)$, conditioned over the three times $T_x(t), T_y(t)$ and $T_z(t)$ as showed above because the homogeneous Poisson process $N$ has independent waiting times.
Furthermore, from the symmetry of the orthogonal motion (both OSM and OUM), for $0\le s\le t$, $X(t)$, with respect to the measure $P\{\cdot\,|\,T_x(t) = s\}$, is equal in distribution to $Y(t)$ with respect to the measure $P\{\cdot\,|\, T_y(t)=s\}$ and similarly with $Z(t)$. Finally, these conditional distributions in (\ref{motoOrtogonaleComponenteAssolutamenteContinua}), in the case of an OUM, would immediately follow from the joint distribution governed by equation (\ref{equazioneTempoPosizioneCongiunta}) and probability (\ref{formaEsplicitaTempoTz}). Unfortunately, as previously explained, we have not been able to obtain an explicit form for the first probability.








\footnotesize{

}

\end{document}